\documentclass[reqno]{amsart}

\usepackage{graphicx, color} 
\usepackage{chngcntr}
\usepackage{apptools}
\usepackage{tikz}
\usepackage{amsmath,amsfonts, amssymb}
\usepackage{epsfig}
\usepackage{wrapfig}
\usepackage{chngcntr}
\usepackage{apptools}

\def\bra{\langle}
\def\ket{\rangle}

\def\N{{\mathbb N}}

\def\R{{\mathbb R}}
\def\S{{\mathbb S}}

\def\cE{{\mathcal E}}

\def\cC{{\mathcal C}}
 
\def\eqnn{\begin{eqnarray*}}
\def\eeqnn{\end{eqnarray*}}
\def\eqn{\begin{eqnarray}}
\def\eeqn{\end{eqnarray}}

\theoremstyle{plain}
\newtheorem{theorem}{Theorem}[section]
\newtheorem{definition}[theorem]{Definition}

\newtheorem{lemma}[theorem]{Lemma}
\newtheorem{corollary}[theorem]{Corollary}
\theoremstyle{remark}
\newtheorem{remark}[theorem]{Remark}

\newtheorem*{thm*}{Theorem}

\theoremstyle{definition}
\newtheorem{defn}{Definition}[section]

\theoremstyle{remark}
\newtheorem{rem}{Remark}[section]

\numberwithin{equation}{section}

\begin{document}

\parskip=8pt
\parindent=0pt

\title[Propagation of $L^\infty$ exp tails]
{\bf{On  pointwise exponentially weighted estimates for the  Boltzmann equation}}

\author[I. M. Gamba] {Irene M. Gamba}
\address{I. M. Gamba,
Department of Mathematics, University of Texas at Austin.}
\email{gamba@math.utexas.edu}

\author[N. Pavlovi\'{c}]{Nata\v{s}a Pavlovi\'{c}}
\address{N. Pavlovi\'{c},  
Department of Mathematics, University of Texas at Austin.}
\email{natasa@math.utexas.edu}

\author[M. Taskovi\'{c}]{Maja Taskovi\'{c}}
\address{M. Taskovi\'{c},
Department of Mathematics, University of Pennsylvania.}
\email{taskovic@math.upenn.edu}

\begin{abstract}
In this paper we prove  a conditional result on the propagation in time of weighted  $L^\infty$  bounds for solutions to the non-cutoff homogeneous Boltzmann equation 
that satisfy propagation in time of weighted $L^1$ bounds. 
To emphasize the general structure of the result  
we express our main result using certain general weights.
We then apply it
to the cases of exponential and Mittag-Leffler weights, for which propagation in time of weighted $L^1$ bounds
is known to hold.
\end{abstract}

\maketitle


\section{Introduction} \label{sec-intro}

The space homogeneous Boltzmann equation 
 \eqn \label{intro-eq-be}
		\partial_t f(t,v) \, = \, Q(f,f)(t,v) 
\eeqn
with $t\in\R^+, v \in\R^d, d\geq 2$, is a mathematical model for the evolution of the probability density $f$ of independent identically distributed particles modeling a rarefied gas with predominantly binary elastic interactions. 
This evolution is governed by a quadratic non-local integral operator $Q(f,f)$ given by
\eqn \label{intro-eq-q}
	 Q(f,f)(t,v) \, = \, \int_{\R^d} \int_{S^{d-1}} \, \big( f' f'_* - f f_* \big) \, B(|u|, \hat{u}\cdot\sigma) \, d\sigma \, dv_*,
\eeqn
\normalsize
where  $f'=f(t,v')$, $f'_*=f(t,v'_*)$, $f_*=f(t,v_*)$, and relative velocity $u$ is given by $u= v-v_*$, with velocities satisfying
$$v'=\frac{v+v_*}{2} +\frac{|u|\sigma}{2}, \qquad 
v'_*=\frac{v+v_*}{2} -\frac{|u|\sigma}{2}.$$
The operator $Q(f,f)$ is called the collision operator,  and is endowed with a collision kernel $B$, modeling the rate of transition states before and after interactions.
Such kernel $B$, given by
\eqn\label{intro-eq-kernel}
	B(|u|, \hat{u}\cdot\sigma) \, = \, |u|^\gamma \, b(\hat{u} \cdot \sigma),
\eeqn
\normalsize
depends on a potential function of the relative speed of the interacting particles $|u|^\gamma$, and on the angle associated to these interactions $b(\hat{u} \cdot \sigma)$.
Hard potentials correspond to positive power growth ($\gamma >0$), while soft potentials correspond to negative ones ($\gamma<0$). 
The collision kernel may or may not be integrable with respect to the angle. The techniques that are needed for estimates in the two cases may differ significantly. In this manuscript we focus on the non-integrable angular kernel, referred to as the {\it non-cutoff} case. For details see Section \ref{sec:bol+mom}.

In this paper we study the propagation in time of weighted $L^\infty$ norms of solutions to the  spatially homogeneous Boltzmann equation in the non-cutoff case for hard potentials.
The stepping stone for this study are properties of the corresponding weighted $L^1$  norms, which have been obtained in the recent work of the authors with Alonso \cite{algapata15}. Since the solution of the Boltzmann equation is a probability distribution, one typically considers  polynomial or exponential weights. The motivation for considering the latter comes from the fact that a Gaussian in the velocity space, $M_\beta(v) = e^{-\beta|v|^2}$, is a stationary state of the Boltzmann equation. Such stationary states are called Maxwellians.

The analysis of $L^1$-weighted norms (called polynomial moments when the weight is a power function of velocity, and exponential moments if the weight function is an exponential) has been developing for several decades, see e.g. \cite{we96, we97, bo97, bogapa04, gapavi09, mo06, alcagamo13, lumo12, algapata15}.
We give more details about previous works on polynomial and exponential-type moments in Section 2. 
Having understood the evolution of  $L^1$-weighted norms of solutions to \eqref{intro-eq-be}-\eqref{intro-eq-kernel}, the natural question is to obtain information about weighted pointwise  bounds. 
The first result in this direction was achieved by Carleman in \cite{ca57}, which was later extended by Arkeryd in \cite{ar83}. Specifically in \cite{ar83} the author proves propagation of  $L^\infty$-polynomially weighted norms for solutions to the spatially homogeneous Boltzmann equation with bounded angular cross-section (which is a special case of the cutoff).  
The first result in the study of propagation of $L^\infty$-exponentially weighted  norms for solutions of the Boltzmann equation for hard potentials with bounded angular kernel was done by Gamba, Panferov and Villani in \cite{gapavi09}, where the authors proved that such solutions are controlled by a Gaussian $M_\beta(v)$ if the initial data is controlled by another Gaussian $M_{\tilde{\beta}}(v)$. 
As in the case of propagation in time of $L^\infty$-polynomialy weighted  norms \cite{ar83}, the propagation of  $L^\infty$-exponentially weighted norms established in \cite{gapavi09} relies on the propagation of $L^1$-exponentially weighted norms. The solutions considered in the works \cite{ar83} and \cite{gapavi09}, whose existence has been established in the earlier work \cite{ar72I}, are such that important physical quantities are controlled (mass and momentum are conserved, and  energy and entropy are bounded).

In view of recent results on the propagation of  $L^1$-exponentially or Mittag-Leffler weighted  norms of solutions in the non-cutoff case \cite{lumo12, algapata15}, a natural question arises as to whether one can propagate  pointwise exponentially weighted bounds in the non-cutoff setting, which is the question that has not been addressed so far and is the object of the study in this manuscript. 
Specifically, in this paper we show propagation in time of   $L^\infty$-{\it weighted}  norms for certain solutions to the homogeneous Boltzmann equation in a non-cutoff case.  The non-cutoff case itself requires a different treatment of the collision operator, since it cannot be split into the gain and loss terms, as is in the case of \cite{gapavi09}. We overcome this difficulty by introducing a new splitting of the collision integral, which is inspired by the splitting typically used in the non-cutoff case (see e.g \cite{aldeviwe00, vi02, si14}). However, our splitting takes into account the  weights.

In order to point out that our propagation in time of  $L^\infty$-weighted  norms relies on propagation in time of  $L^1$-weighted   norms, we  state the main result in terms of certain {\it general} weights, see Theorem \ref{thm-1}.  This can be understood in a spirit of an important step of the De-Giorgi-Nash-Moser type argument, in the sense that $L^{\infty}$ bounds depend on  $L^1$ bounds.
Consequently we 
apply the main theorem result to cases of exponential and Mittag-Leffler weights, for which propagation in time of $L^1$-weighted norms 
holds, see Corollary \ref{cor}. 

Some of the tools that we use to prove the main theorem are motivated by the regularity theory of integro-differential equations, brought to the context of the Boltzmann equation in the recent work of Silvestre \cite{si14}, where the author showed propagation in time of $L^\infty$ bounds, {\it without weights}, for classical solutions of the non-cutoff Boltzmann equation. Those are solutions that satisfy conservation of mass, momentum and energy, as well as entropy decay, whose existence has not been established yet.

Our result on propagation of $L^\infty$-weighted norms of solutions to the homogeneous non-cutoff Boltzmann equation that propagate $L^1$-weighted norms is conditional.  Namely, although the concept of weak solutions  is sufficient to prove propagation of $L^1$-polynomially or exponentially weighted norms \cite{we97,  algapata15}, in this paper we had to work with a stronger type of a solution whose existence is not yet established.  We prove a priori results for such solutions. For more details, see Theorem \ref{thm-1} and remarks afterwards.

{\bf Organization of the paper:} 
In Section \ref{sec:bol+mom}
we recall the Boltzmann equation, weak solutions, and existing results on polynomial and exponential moments. In Section \ref{sec:main}
we give precise statement of our main result in Theorem \ref{thm-1} and an application in Corollary \ref{cor}. In Section \ref{sec:prev} we review previous results on $L^\infty$ bounds which will be relevant for the proof of our main result. In Section \ref{sec:proof} we give a proof 
of Theorem \ref{thm-1}. Section \ref{sec:ex} provides examples of weight functions that can be used in Theorem \ref{thm-1}, as well as the proof of Corollary \ref{cor}.

{\bf{Acknowledgements.}} 
This work of I.M.G. has been supported by NSF grants DMS-1413064 and NSF-DMS-RNMS-1107465, the one of N.P. has been supported by NSF grants DMS-1516228 and DMS-1440140 and the one of M.T.  has been supported by
NSF grants DMS-1413064, NSF-DMS-RNMS-1107465 and DMS-1516228. 
The authors also gratefully acknowledge support from the Institute of Computational Engineering and Sciences (ICES) at The University of Texas at Austin and the Mathematical Sciences Research Institute (MSRI) in Berkeley, California.

\section{The Boltzmann equation and its moments} \label{sec:bol+mom}

\subsection{The Boltzmann equation}

We consider the Cauchy problem for the spatially homogeneous Boltzmann equation
 \eqn \label{eq-be}
\left\{
		\begin{array}{l}
			\partial_t f(t,v) \, = \, Q(f,f)(t,v),  \quad t\in\R^+, v\in\R^d, \quad d\geq 2 \vspace{5pt}\\ 
			f(0,v) \, = \, f_0(v).
		\end{array}
	\right. 
\eeqn
\normalsize
for time $t\in\R^+$ and velocity $v\in\R^d$. It is an evolution equation of the density $f(t,v)$ of particles in a rarefied gas. The operator $Q(f,f)$, called the collisional operator, measures the change of $f$ due to instantaneous binary collisions of particles. It is a quadratic integral operator defined via
\eqn \label{eq-q}
	 Q(f,f)(x,t,v) \, = \, \int_{\R^d} \int_{S^{d-1}} \, \big( f' f'_* - f f_* \big) \, B(|u|, \hat{u}\cdot\sigma) \, d\sigma \, dv_*.
\eeqn
\normalsize

We employ abbreviated notation $f'=f(x,t,v')$, $f'_*=f(x,t,v'_*)$, $f_*=f(x,t,v_*)$  often used in the context of Boltzmann equation.

\begin{wrapfigure}{r}{0.35 \textwidth}
\setlength{\unitlength}{1cm}
\begin{tikzpicture}[scale=0.4]
\centering
\begin{scope}
   	\draw (0,0) circle (4);
	\draw [->] (-4,0)-- (4,0);
	\node at (-4.5, -0.5) {$v_*$};
	
	\node at (3.7, 0.3)  {$u$};
	\node at (2.4, 2.5)  {$u'$};
	
	\draw [->] [thick] (0,0)--(2,0);
	\node at (2, .4)  {$\hat{u}$};
	\node at (4.2, -0.5) {$v$};
	\node at (3.5, 2.8) {$ v'$};

	\draw [->] (-3,-2.6) -- (3,2.6);
	\node at (-3.3, -3.3)  {$v'_*$};

	\draw [->] [thick] (0,0)--(1.5, 1.3);
	\node at (1.5, 1.6)  {$\boldsymbol{\sigma}$};

	\draw[->] [dashed] (2,-6)--(4,0);
	\draw[->] [dashed] (2,-6)--(-4,0);
	\draw[->] [dashed] (2,-6)--(-3,-2.6);
	\draw[->] [dashed] (2,-6)--(3,2.6);
	\node at (2, -6.5) {O};
\end{scope}
\centering
\end{tikzpicture}
\vspace{-0.5in}
\end{wrapfigure}

Vectors $v', v'_* \in \R^d$ stand for velocities of a pair of particles before the collision, while $v, v_* \in \R^d$ denote corresponding post-collisional velocities. Due to the conservation of momentum ($v+v_* = v' + v'_*$) and energy ($|v|^2 + |v_*|^2 = |v'|^2 + |v'_*|^2$), these pairs of velocities are related via
\begin{flalign*}
	& v' \, = \,  \frac{v+v_*}{2} \, + \,  \frac{|v-v_*|}{2}\; \sigma, \\
	& v'_* \, = \, \frac{v+v_*}{2} \, - \,  \frac{|v-v_*|}{2} \; \sigma, 
\end{flalign*}
\normalsize
where $\sigma  \in \S^{d-1}$ is the unit vector in the direction of the pre-collisional relative velocity $u' = v' -v'_*$.
The relative post-collisional velocity is denoted by $u=v-v_*$, and the unit vector with the same direction by $\hat{u} :=u/|u|$.

Most important information about collisions is encoded in the collisional kernel $B(|u|, \hat{u}\cdot\sigma)$, assumed to  take the factorized form that separates kinetic and angular parts
\eqn\label{eq-kernel}
	B(|u|, \hat{u}\cdot\sigma) \, = \, |u|^\gamma \, b(\hat{u} \cdot \sigma).
\eeqn
\normalsize
We write $\hat{u} \cdot \sigma = \cos\theta$, where $\theta \in [0,\pi]$ is the angle between the pre and post collisional relative velocities (see the Figure 1). With an abuse of notation, we also denote the kernel $B(|u|, \hat{u}\cdot\sigma)$ as $B(|u|, \theta)$. 

In this manuscript, we study the variable hard potentials case
\begin{align}\label{eq-gamma}
0<\gamma\leq 1.
\end{align}

In many models, the angular kernel  $b(\hat{u} \cdot \sigma)$, which is a positive measure over the sphere $\S^{d-1}$, is not integrable. However, since Grad's work \cite{gr63} in 1963, integrability is often assumed 
to simplify the analysis of the collisional operator since under this assumption the operator $Q$ can be split into the gain $Q^+$ and loss $Q^-$ terms, which then can be analyzed separately. More precisely, 
in that case one has: 
\begin{align} \label{Q-split}
Q(f,f)\, 
& =\, Q^+(f,f) - Q^-(f,f) \\ \nonumber
& : =  \int_{\R^d} \int_{S^{d-1}} \, f'  f'_*  \, 
	B(|u|, \hat{u}\cdot\sigma) \, d\sigma \, dv_* 
	- f\, \int_{\R^d} \int_{S^{d-1}} \,  f_* \, 
	B(|u|, \hat{u}\cdot\sigma) \, d\sigma \, dv_*.
\end{align}
The hope was that removing the singularity of the angular kernel should not affect properties of the equation. However, it has been observed recently (e.g. \cite{li94}, \cite{de95}, \cite{dego00}, \cite{dewe04}) that the singularity of $b(\cos\theta)$ carries a regularization. This, and the analytical challenge, motivated further study of the non-cutoff regime, which is the setting we consider here.

More precisely, inspired by the inverse power law model, in this paper we consider the following non-cutoff model with
\begin{equation}\label{b-pw}
b(\cos{\theta}) \approx  (\sin{\theta})^{-(d-1) - \nu},  \quad \mbox{with} \;\nu \in (0,2).
\end{equation}
The symbol $a \approx b$  is understood in the following sense: there are universal constants $c_1, c_2$ so that $c_1 \, b \le a \le c_2 \, b$. For the range of $\nu$ under consideration, i.e.  $\nu \in (0,2)$, the function $b(\cos\theta)$ indeed is not integrable over the unit sphere (it is integrable for $\nu<0$).  However, if it is weighted with $(\sin\theta)^{\nu+}$, it becomes integrable.

\begin{remark}
In the particular case of the inverse power law model in $3$ dimensions, parameters $\gamma$ and $\nu$ have the following formulas
\begin{align}\label{ipl}
\gamma = \frac{s - 5}{s-1}, \quad \nu = \frac{2}{s-1},
\end{align}
where $s$ is a parameter strictly larger than $2$. Note that, indeed, in this model $\nu \in (0,2)$. In addition, variable hard potentials correspond to $s>5$, which implies $\nu < \tfrac{1}{2}$.

\end{remark}

Due to symmetries of the collisional kernel $Q(f,f)$, its value remains the same if $B$ is replaced with $\tilde{B}$, provided that
$$
B(|u|, \theta) \, + \, B(|u|, \theta + \pi) \, = \, \tilde{B}(|u|, \theta) \, + \,\tilde{ B}(|u|, \theta + \pi).
$$
In the case when both $B$ and $\tilde{B}$ are factorized, i.e.  $B(|u|, \theta) = |u|^\gamma \, b(\theta)$  and $\tilde{B}(|u|, \theta) = |u|^\gamma \, \tilde{b}(\theta)$ with the same parameter $\gamma$, then this condition reduces to 
\begin{align}\label{bcond}
b(\theta) \, + \, b(\theta + \pi) \, = \, \tilde{b}(\theta) \, + \, \tilde{b}(\theta + \pi).
\end{align}
Given $b(\theta)$ as in \eqref{b-pw}, there are many ways to construct $\tilde{b}$ that satisfies \eqref{bcond}. A frequent choice is to set
\begin{align*}
\tilde{b}(\cos\theta) = 
\left\{
	\begin{array}{ll}
			2b(\cos\theta) , \quad &\mbox{if} \,\, \cos\theta>0\vspace{5pt}\\ 
			 0, \quad &\mbox{if} \,\, \cos\theta<0,
	\end{array}
	\right. 
\end{align*}
thus reducing the support of the angular kernel to the right half of the sphere. In this manuscript, however, we will use the following behavior on half spheres, as was the case in \cite{si14} 
\begin{align}\label{eq-b}
\tilde{b}(\cos\theta) \approx 
\left\{
	\begin{array}{ll}
			 |\sin\theta|^{-(d-1)-\nu}, \quad &\mbox{if} \,\, \cos\theta>0\vspace{5pt}\\ 
			 |\sin\theta|^{1 + \gamma +\nu}, \quad &\mbox{if} \,\, \cos\theta<0.
	\end{array}
	\right. 
\end{align}
This particular choice is tailored for the proof of Lemma \ref{si2}. We provide more details about this in Appendix \ref{ap-b}. 
From now on we will abuse the notation and write 
$b(\cos\theta)$ instead of   $\tilde{b}(\cos\theta)$.


\subsection{Weak solutions}

In this section, we recall the definition of a weak solution, whose existence in three dimensions and in the non-cutoff case \eqref{eq-nc} with $\beta \in (0,2]$ is proved in \cite{ar81, vi98, go97}. For more existence results in the non-cutoff regime, see for example \cite{vi98, alvi02, lumo12, uk84, grst11, amuxy11, amuxy12}
\begin{defn}\label{defn-weak}
Let $f_0 \geq 0$ be a function defined in $\R^d$ with finite mass, energy and entropy
\begin{align}\label{eq-entropy}
\int_{\R^d} f_0(v) \, \left( 1 + |v|^2 + \log(1+f_0(v))\right) \, dv < +\infty.
\end{align}
Then we say $f$ is a {\it weak solution} to the Cauchy problem \eqref{eq-be} if it satisfies the following conditions:
\begin{itemize}
\item $f\geq 0, \; f \in C(\R^+; \mathcal{D}'(\R^d)) \, \cap \, L^1([0,T]; L^1_{2+\max\{\gamma,0\}})$
\item $f(0,v) = f_0(v)$
\item $\forall t\geq 0$: \; $\int f(t,v) \psi(v) dv = \int f_0(v) \psi(v) dv$, for $\psi(v) = 1, v_1, ..., v_d, |v|^2$
\item $f(t,\cdot) \in L\log L$ and $ \forall t\geq 0: \, \int f(t,v) \log f(t,v) dv \leq \int f_0(v) \log f_0 dv$
\item $\forall \phi(t,v) \in C^1(\R^+, C^\infty_0(\R^3))$, $\forall t \geq 0$ we have that 
\begin{align*}
\int_{\R^d} f(t,v) \phi(t,v) dv \,  - \, \int_{\R^d} f_0(v) \phi(0,v) dv 
\,& - \, \int_0^t d\tau \int_{\R^d} f(\tau, v) \partial_\tau \phi(\tau,v) dv\\
& = \int_0^t d\tau \int_{\R^d} Q(f,f)(\tau,v) \phi(\tau,v) dv.
\end{align*}
\end{itemize}

\end{defn}

\subsection{Polynomial and exponential moments} As announced in the introduction, our main result on  the propagation in time of  weighted $L^\infty$ bounds is achieved by exploiting propagation of the corresponding   weighted $L^1$ bounds.  In this section, we recall what is known at the level of  weighted $L^1$ bounds. Before we review results in this direction, we recall the definition of polynomial and exponential moments.

Solutions to the Boltzmann equation are probability density functions  $f(t,v)$. Therefore their polynomially weighted  $L^1$ norms, i.e. its statistical moments or observables, play a significant role for further studies of the solution behavior. One can also study more general moments. Since the equilibrium state of the Boltzmann equation is a Maxwellian distribution, i.e. a Gaussian distribution in velocity space, we are particularly interested in the study of the so-called exponential moments, i.e. exponentially weighted $L^1$ norms.

\begin{definition}[Polynomial and exponential moments]
 \label{def-mom}
Polynomial moment of order $q$ of a function $f(t,v)$ is defined via
\begin{align}
	& m_q(t) \, := \, \int_{\R^d} f(t,v) \;  \bra v \ket^q \; d(v), \label{def-polymom}
	\end{align}
Exponential moment of order $s$ and rate $\alpha$ of a function $f(t,v)$ is defined by
\begin{align}
\mathcal{M}_{\alpha, s}(t) := \int_{\R^d}  f(t,v) \; e^{\alpha \, \bra v \ket^s} \; dv. \label{def-expmom}
\end{align}

\end{definition} 

In an extensive work  including e.g. 
\cite{el83,de93, we96, we97,miwe99} 
generation of polynomial moments 
\begin{align}\tag{G-poly-1}\label{eq-g}
& \int_{\R^d}  f(0,v) \; \langle v \rangle^2\; dv < C_0, \\
& \qquad \Rightarrow \qquad
\forall q>2, \; \forall t_0>0, \; \exists C>0, \; \forall t\geq t_0: \;\; \int_{\R^d}  f(t,v) \; |v|^q \; dv < C. \nonumber
\end{align}
and propagation of polynomial moments 
\begin{align}\tag{P-poly-1}\label{eq-p}
& \int_{\R^d} f(0,v) \; |v|^q dv < C_0 
\;\; \mbox{for some} \;\; q>0\\
& \qquad \Rightarrow \qquad
\exists C>0, \;\; \forall t\geq 0: \;\; \int_{\R^d}  f(t,v) \; |v|^q \; dv < C. \nonumber
\end{align}
was shown, both for Grad's cutoff and the non-cutoff case.

Propagation of exponential moments
 \begin{align}\tag{P-exp-1}\label{eq-p}
& \int_{\R^d} f(0,v) \; e^{\alpha_0 \, \bra v \ket^s} dv < C_0 
\;\; \mbox{for some} \;\; \alpha_0, s>0\\
& \qquad \Rightarrow \qquad
\exists C>0, \;\; \exists  0<\alpha \leq \alpha_0, \;\; \forall t\geq 0: \;\; \int_{\R^d}  f(t,v) \; e^{\alpha \, \bra v \ket^s} \; dv < C. \nonumber
\end{align}
and generation of exponential moments 
\begin{align}\tag{G-exp-1}\label{eq-g}
& \int_{\R^d}  f(0,v) \; \langle v \rangle^q\; dv < C_0, \;\;\mbox{for some} \;\; q>2 \\
& \qquad \Rightarrow \qquad
\exists s, \alpha,  C>0, \;\; \forall t>0: \;\; \int_{\R^d}  f(t,v) \; e^{\alpha \, \bra v \ket^s} \; dv < C. \nonumber
\end{align}
was studied later, first under the Grad's cutoff assumption. The study was initiated by Bobylev \cite{bo84, bo97}, where the fundamental connection with polynomial moments was exploited. Namely, Taylor series expansion of $e^{\alpha \langle v \rangle^s}$ yields the following representation of exponential moments as an infinite sum of renormalized polynomial moments 
\eqn\label{def-expmomsum}
	\mathcal{M}_{\alpha, s}(t) \; = \;\sum_{q=0}^\infty \frac{m_{qs}(t) \; \alpha^q}{q!}.
\eeqn 
This was further developed for example by Bobylev, Gamba, Panferov in \cite{bogapa04}, Gamba, Panferov, Villani in \cite{gapavi09}, and Mouhot \cite{mo06}. 
All these papers used a technique based on 
establishing a term-wise geometric decay for terms in \eqref{def-expmomsum}.
Recently a new type of proof was developed in  the work of Alonso, Canizo, Gamba, Mouhot \cite{alcagamo13},  where estimates on the partial sums corresponding to  \eqref{def-expmomsum} were obtained. On the other hand, the non-cutoff case in the context of exponential moments was considered only recently by Lu and Mouhot \cite{lumo12}, where the  authors established the generation of exponential moments up to the order $s \in (0, \gamma]$, by implementing the
term-by-term method:

Recently, the authors of this paper together with Alonso \cite{algapata15} extended the result \cite{lumo12} to the exponential tails of order higher that $\gamma$, 
by implementing the partial sum method of \cite{alcagamo13} in the non-cutoff setting. 
To exploit decay of certain sums of Beta functions, our calculations led to expressions  similar to \eqref{def-expmomsum}, which in place of $q!$ have $\Gamma (aq + 1)$, with non-integer $a>1$. These new sums are associated to Mittag-Leffler functions, which are a generalization of the Taylor expansion of the exponential function. They are defined for some parameter $a>0$ by
\eqn\label{def-ml}
	\cE_a (x) := \sum_{q=0}^\infty \frac{x^q}{\Gamma(aq+1)}.\label{eq-mlfc}
\eeqn
\normalsize
It is well known (see e.g. \cite{emot53}) that the Mittag-Leffler function $\cE_a$
asymptotically behaves like an exponential function of order $1/a$ 
\eqn\label{eq-asym}
 \cE_{2/s} (\alpha^{2/s} \, x^2) \sim  e^{\alpha \, x^s}, \qquad \mbox{for} \,  x \rightarrow \infty,
\eeqn
and consequently, for some constants $c, C$ that depend on $\alpha, s$
\begin{align}\label{eq-asym1}
 c \, e^{\alpha \, x^s} \, \le \,\cE_{2/s} (\alpha^{2/s} \, x^2) \, \le \, C \, e^{\alpha \, x^s}.
\end{align}

This motivated our definition of Mittag-Leffler moments in \cite{algapata15}, whose finiteness still describes exponential tail behavior in $L^1$ sense.

\begin{definition}[Exponential and Mittag-Leffler moments] \label{def-mom-ML} 
Mittag-Leffler moment of  order $s$ and rate $\alpha>0$ of a function $f$ is introduced via
\eqn
	\displaystyle\int_{\R^d}  f(t,v) \;\; \cE_{2/s} (\alpha^{2/s} \, \bra v \ket^2 ) \; dv.
\eeqn
\normalsize
\end{definition}

Before we give the precise statement of the result in \cite{algapata15}, we give the condition on the angular kernel empoyed in \cite{algapata15}
\begin{align}\label{eq-nc} 
\int_{\S^{d-1}} b(\hat{u} \cdot \sigma) \sin^{\beta}{\theta}  \; &d\sigma = V_{d-2}  \int_0^\pi b(\cos{\theta})   \; \sin^{ \beta}{\theta} \; \sin^{d-2}{\theta}\; d\theta < \infty,
\end{align}
\normalsize
for some $\beta \in (0,2]$.  Here $V_{d-2}= \frac {\pi^{(d-2)/2}}{\Gamma((d-1)/2)}$ is the volume of the $d-2$ dimensional unit sphere. Note that the cross section  \eqref{b-pw} that is considered in this paper satisfies \eqref{eq-nc} with any $\beta > \nu$. 
In what follows we also use the notation
  $$L^1_k = \{f \in L^1(\R^d): \int_{\R^d} f \langle v \rangle^{k} dv <\infty\}.$$
Now, we are ready to recall the main result from the earlier work of authors with Alonso \cite{algapata15}, which will be used in the application of the main result of this paper (see Corollary \ref{cor}).

\begin{theorem}[Generation and Propagation of Mittag-Leffler moments]\label{thm}
Suppose $f$ is a solution to the Boltzmann equation \eqref{eq-be} associated to the  initial data $f_0 \in L^1_{2}$. Suppose the collision kernel is of the form \eqref{eq-kernel} with $0 < \gamma \leq 1$.
\vspace{-0,1in}
\begin{itemize}
\item[(a)] (Generation of exponential moments) If  the angular kernel satisfies the non-cutoff condition \eqref{eq-nc} with $\beta = 2$, then the exponential moment of order $\gamma$ is generated with a rate $r(t)={\alpha \, \min\{t,1\}}$. More precisely, there are positive constants $C, \alpha$, depending only on $b$, $\gamma$ and initial mass and energy, such that
\begin{equation}\label{eq-gener}
	\int_{\R^d}  \, f(t,v) \; e^{\alpha \, \min\{t,1\} \, |v|^\gamma} \, dv \, \leq \, C, \quad \mbox{for} \;\; t\geq 0.
\end{equation}
\item[(b)] (Propagation of Mittag-Leffler moments) Let $s \in (0, 2)$ and suppose that the Mittag-Leffler moment of order $s$ of the initial data $f_0$ is finite with a rate $r=\alpha_0$, that is,
\begin{equation}\label{eq-id}
	\displaystyle\int_{\R^d}  f_0(v) \;\; \cE_{2/s} (\alpha_0^{2/s} \, \bra v \ket^2 ) \; dv < M_0. 
\end{equation}
Suppose also that the angular cross-section satisfies assumption 
\begin{align}\label{s_condition}
	& \mbox{with} \;\; \beta=2,    &  \mbox{if} \;\;s\in (0, 1] \nonumber \\
	&\mbox{with} \;\; \beta = \frac{4}{s}-2,   & \mbox{if} \;\; s\in (1, 2).
\end{align}
Then, there exist positive constants $C, \alpha$, depending only on $M_0$, $\alpha_0$, $b$, $\gamma$ and initial mass and energy such that the Mittag-Leffler moment of order $s$ and rate $r(t) =\alpha$ remains uniformly bounded in time, that is 
\begin{equation}\label{eq-prop}
	\displaystyle\int_{\R^d}  f(t,v) \;\; \cE_{2/s} (\alpha^{2/s} \, \bra v \ket^2 ) \; dv < C, \quad \mbox{for} \; \; t\geq 0.
\end{equation}
\end{itemize}
\end{theorem}

\begin{remark}
Since Mittag-Leffler function asymptotically behaves like an exponential function \eqref{eq-asym1}, finiteness of exponential moment of order $s$ is equivalent to finiteness of the corresponding Mittag-Leffler moments. Hence, in fact, classical exponential moments are propagated in time too.
\end{remark}


\section{The main result}\label{sec:main}

In this section we state our main result - the propagation in time of certain  weighted $L^{\infty}$ bounds of solutions to the homogeneous Boltzmann equation in the non-cutoff setting. This result holds for exponential and Mittag-Leffler weight functions, and in both cases the proof relies on the corresponding weighted $L^1$  bounds. To emphasize this, and to make the presentation clear, we state the result for a general weight function, which is introduced to mimic the exponential asymptotic behavior. In particular, the weight $w$, which is a function of velocity $v$, is introduced depending on two parameters $\alpha > 0$ and $p\in (0,2]$. One can think of $\alpha$ and $p$ as describing the exponential behavior $e^{\alpha \langle v \rangle^{p}}$. More precisely, we assume that the weight function $w(v; \alpha, p)$ has the following properties:
\begin{enumerate}
\vspace{-0.1in}
	\item[(P1)] $w(v; \alpha, p)$ is strictly positive, radially increasing in $v$, and increasing in $\alpha$.
	\item[(P2)] For  every $\alpha, \alpha', p >0$
		there exists a constant $C = C(\alpha, \alpha', p)$  and  
		$c_2= c_2(p)$, so that for every $v\in \R^d$ 
		 	$$ w(v;\alpha, p) \; w(2 v; \alpha', p) \; 
			\leq \; C \, w( v; \alpha + c_2 \alpha', p).$$
	\item[(P3)] Given $\delta \in [0,1]$, and $\alpha, \alpha', p >0$ and $k\geq 0$, there exist constants $C=C(\delta,  k, \alpha, \alpha', p )$ and $D= D(\delta,  k, \alpha, \alpha', p )$  so that  $\forall v\in \R^d$ 
\begin{align}
& \mbox{If} \, \, \delta \alpha < \alpha',\quad \mbox{then} \,\,
	\frac{w(v; \alpha, p)^\delta}{w(v; \alpha', p)} \leq \frac{C}{\langle v \rangle^k} \label{P3a}\\
& \mbox{If} \, \, \delta \alpha > \alpha',\quad \mbox{then} \,\,
	\frac{w(v; \alpha, p)^\delta}{w(v; \alpha', p)} \geq D \, \langle v \rangle^k.  \label{P3b}
\end{align}
	\item[(P4)] For every $\alpha, p >0$ there is a constant $C = C(\alpha, p)$, 
		so that $\forall v \in \R^d$
			$$\left| \nabla_v\left( \frac{1}{w (v; \alpha, p)} \right) \right|\le C \langle v \rangle.$$
\end{enumerate}

\medskip

%

\begin{theorem}{(Propagation of $L^\infty_w$ tails)}\label{thm-1}\\
 Consider the Cauchy problem \eqref{eq-be} with the cross section  \eqref{eq-kernel} with $0 < \gamma \leq 1$, and the angular kernel \eqref{eq-b} with $\nu \in (0,1]$, and the initial data $f_0(v)$ which has finite mass, energy and entropy \eqref{eq-entropy}. 
Let $\alpha_0>0$, $p \in (0,2]$ and let $w( v; \alpha_0, p) \in C(\R^d \times \R^+ \times \R^+)$ be a weight function that satisfies properties $(P1-P4)$.  Suppose $f(t,v)$ is a continuous function  in $(t,v)$ such that  
\begin{itemize}
\item[(i)] for every $0<t<T$, \, $f \in L^\infty([t, T]; \mathcal{S}(\R^d))$
\item[(ii)]  $m_{\alpha_0,p}(t):=\|f(t,v) \; w(v, \alpha_0, p)\|_{L^\infty_v}$ is continous in $t$, finite for all $t \in \R^+$, and for every $t \in \R^+$ the norm is attained at some velocity $v$.
\end{itemize}

In addition, suppose that for every $\alpha > 0$ there exists $0<\alpha_1 < \alpha$ and a constant $C_1>0$  (uniform in time) such that
\begin{itemize}
\item[{(iii)}]
\begin{align*}
& \mbox{ if} \;\;  \| f_0(v) \, w( v; \alpha, p) \|_{L^1_v} < \infty, \nonumber \\
&\qquad \mbox{then} \;\; \| f(t,v) \, w( v; \alpha_1, p) \|_{L^1_v} < C_1, \quad  \forall t\geq 0. 
\end{align*}

\end{itemize}%
Then there exists $0<\alpha_2<\alpha_0$ and a constant $C$  (uniform in time, depending only on $C_1$, $p, \alpha_0$, initial data and the cross section) such that
\begin{align}
& \mbox{ if} \;\;  \| f_0(v) \, w( v; \alpha_0, p) \|_{L^\infty_v} < \infty, \nonumber \\
&\qquad \mbox{then} \;\; \| f(t,v) \, w( v; \alpha_2, p) \|_{L^\infty_v} < C, \quad  \forall t\geq 0. \label{pw}
\end{align}
In particular,
\begin{align} \label{l1-linfty}
\| f(t,v) \, w( v; \alpha_2, p) \|_{L^\infty_v} \leq c \| f(t,v) \, w( v; \alpha_1, p) \|_{L^1_v} \leq C.
\end{align}
\end{theorem}

\medskip

\begin{remark}
Before we discuss corollary and specific wheights that can be used, we first address the assumtions made in the theorem:
\begin{itemize}
\item[(0)]  In the assumtion (ii), we emphasize that even though we assume that  $\|f(t,v) \; w(v, \alpha_0, p)\|_{L^\infty_v}$ is finite for all $t$, we do not assume that the bound is uniform. The point of the theorem is precisely in showing the uniform in time bound of the norm 
$\|f(t,v) \; w(v, \alpha_2, p)\|_{L^\infty_v}$.

\item[(1)] In the case of hard potentials that we consider, condition (i) is satisfied. Namely, Alexandre, Morimoto, Ukai, Xu and Yang \cite[Theorem 1.2]{amuxy12} proved that weak solutions are in fact of Schwartz class provided that polynomial moments ($L^1$ polynomially weighted  norms of solutions) of all orders remain finite. This condition is known to be satisfied. In fact, for hard potentials exponential moment of order $\gamma$ (that is, $L^1$ $e^\alpha \langle v\rangle^\gamma$-weighted  norm) is generated instantaneously  and remains uniformly bounded in time. Therefore, weak solutions are really of the Schwartz class locally in time.

\item[(2)] Assumption (ii) enables us to apply a modification of some of the techniques of Silvestre \cite{si14}. Whether one can prove existence of solutions that satisfies condition (ii) is an open question.

\item[(3)] In the case of hard potentials and exponential or Mittag-Leffler weights, the assumption (iii) is known to be true \cite{algapata15}.

\end{itemize}
\end{remark}

In Section \ref{sec:ex} we provide examples of weight functions $w$ that satisfy properties (P1)-(P4), including  exponential  and Mittag-Leffler functions. For these functions, it has already been established that corresponding moments (i.e. weighted $L^1$ bounds) propagate in time, and thus satisfy the assumption (iii) of the Theorem \ref{thm-1}. As a consequence, we will be able to prove the following statement.

\begin{corollary}{(Exponential and Mittag-Leffler $L^\infty$ bounds)}\label{cor}
Consider the Cauchy problem \eqref{eq-be}  with the cross section \eqref{eq-kernel} with $0 < \gamma \leq 1$ and the angular kernel satisfying \eqref{eq-b} with $\nu \in (0,1]$, and the initial data $f_0(v)$ which has finite mass, energy and entropy.  Suppose $f(t,v)$ is a continuous function  in $(t,v)$.
\begin{itemize}
\item[(a)]  If inital data satisfies $f_0(v) \leq C_0 e^{-\alpha_0 \langle v \rangle^p}$ for some $\alpha_0>0$ and $p< \frac{4}{\nu + 2}$, and if 
\begin{itemize}
\item[(i)] for every $0<t<T$, \, $f \in L^\infty([t, T]; \mathcal{S}(\R^d))$
\item[(ii)]  $ m_{\alpha_0,p}(t):=\|f(t,v) \;e^{-\alpha_0 \langle v \rangle^p}\|_{L^\infty_v}$ is continous in $t$, finite for all $t \in \R^+$, and for every $t \in \R^+$ the norm is attained at some velocity $v$,
\end{itemize}

\noindent then there exist $0<\alpha < \alpha_0$ and a constant $C >0$   (uniform in time, depending only on $C_0$, $p, \alpha_0$, initial data and the cross section), so that 
$$f(t,v) \leq C e^{-\alpha \langle v \rangle^p}, \qquad \mbox{for all} \;\; t \ge 0.$$
\item[(b)]  On the other hand, if initial data satisfies $f_0(v) \leq C_0 \cE_{2/p}(\alpha_0^{2/p} \langle v \rangle^2)$ for some $\alpha_0>0$ and $p< \frac{4}{\nu + 2}$, and if 
\begin{itemize}
\item[(i)] for every $0<t<T$, \, $f \in L^\infty([t, T]; \mathcal{S}(\R^d))$
\item[(ii)]  $ \|f(t,v) \; \cE_{2/p}(\alpha_0^{2/p} \langle v \rangle^2)\|_{L^\infty_v}$ is continous in $t$, finite for all $t \in \R^+$, and for every $t \in \R^+$ the norm is attained at some velocity $v$,
\end{itemize}
 
\noindent  then there exist $0<\alpha < \alpha_0$ and a constant $C >0$   (uniform in time, depending only on $C_0$, $p, \alpha_0$, initial data and the cross section), so that
 $$f(t,v) \leq C \cE_{2/p}(\alpha^{2/p} \langle v \rangle^2), \qquad \mbox{for all} \;\; t \ge 0.$$

\end{itemize}

\end{corollary}



\section{Previous results on $L^\infty$ bounds} \label{sec:prev}
In this section we review recent $L^\infty$ bounds and weighted  $L^\infty$ bounds on solutions to the homogeneous Boltzmann equation, which will be relevant for the proof of our main result.

\subsection{Towards $L^\infty$ bounds: Carleman representation}
In previous works \cite{ar83, gapavi09, si14}
on enhancing upper $L^1$ bounds to upper $L^\infty$ bounds
of solutions to the homogeneous Boltzmann equation, 
a specific change of variables was used, which is often referred to as Carleman representation. 
This technique was developed by Carleman \cite{ca57}. See also \cite{we94, gapavi09, grst11}. In this process the integration over the $(d-1)$ dimensional sphere reduces to the integration over a hyperplane $\Pi$ that is orthogonal to $v'-v$. This is achieved via replacing variables $(v,v_*, \sigma)$ by $(v,v', w)$, where $w$ belongs to a hyperplane $\Pi$. 

In this manuscript we will use the following version of Carleman representation:

\begin{lemma}[Carleman representation, \cite{gapavi09, grst11, si14, we94a, ca57}]\label{car}
Let $H: \R^d \times \R^d  \rightarrow \R$. Then
\begin{align}\label{Carl}
\int_{\R^d} \int_{\mathcal{S}^{d-1}} H(v,v') \; f(v'_*) \;  B(r, \theta) \; d\sigma dv_* 
\; = \;  \int_{\R^d} H(v,v') \; K _f(v,v') \; dv',
\end{align}
where the kernel $K_f(v,v')$ is given by
\begin{align}\label{kf}
K_f(v,v') \; = \; \frac{2^{d-1}}{|v'-v|} \; \int_{\{w: w\cdot (v'-v) = 0\}} f(v+w) \; B(r,\theta) \; r^{-d+2} \; dw.
\end{align}
In the new set of variables $(v,v',w)$, we have
\begin{align*}
r = \sqrt{|v'-v|^2 + |w|^2}, \; \; \cos{\frac{\theta}{2}} = \frac{|w|}{r}, \\
v'_* = v + w, \; \; v_* = v' +w.
\end{align*}
\end{lemma}

\subsection{Weighted $L^\infty$ bounds for the homogeneous Boltzmann equation} 

Once weighted $L^1$  estimates are developed, the next important question is understanding pointwise behavior of solutions. 
This has been achieved in the cutoff case for the polynomial weights by Arkeryd \cite{ar83} and for exponential weights  in the work of Gamba, Panferov and Villani \cite{gapavi09}. 
We give the full statement from \cite{gapavi09} of the propagation in time of exponentially weighted $L^{\infty}$ norms of solutions to the homogeneous Boltzmann equation below.

\begin{theorem}[Gaussian-weighted pointwise bounds, cutoff case, \cite{gapavi09}]
Consider the Cauchy problem \eqref{eq-be}, \eqref{eq-kernel}, for the hard potentials $0 < \gamma \le 1$ with the angular kernel satisfying $0 \le b(\cos\theta) \le c \sin^\alpha\theta$, with $\alpha <d-1$, which corresponds to a Grad's cutoff. Suppose $f(t,v)$ is the unique solution to this Cauchy problem  with initial data satisfying 
$$ 0 \le f_0(v) \le e^{-a_0 |v|^2 + c_0}, \qquad \mbox{for a.e.} \; v\in \R^d, \; \mbox{for all} \; t \ge 0$$
that conserves the initial mass and energy. Then there exist constants $a>0$ and $c \in \R$ so that 
$$ f(t,v) \le e^{-a|v|^2 +c},  \qquad \mbox{for a.e.} \; v\in \R^d, \; \mbox{for all} \; t \ge 0.$$
\end{theorem}

The key tool for proving the pointwise estimate of \cite{gapavi09} is the comparison principle for
a solution to the spatially homogeneous Boltzmann equation, which was 
also established in \cite{gapavi09}, thanks to a monotonicity property 
of a linear Boltzmann semigroup. A crucial ingredient for a successful application of the 
comparison principle is an exponentially weighted upper bound of the linear ``gain" operator, 
which was obtained in \cite{gapavi09} using Carleman's form of the ``gain" term
and careful estimates some of which use the propagation of exponentially weighted 
$L^1$ norms of the solution. 

Although the comparison principle of \cite{gapavi09} 
is stated in the case of a cutoff, the proof suggests that it should be expected in a non-cutoff case. However
that is not sufficient to obtain the analogue of the point-wise propagation estimate of \cite{gapavi09}
in a non-cutoff case, 
since in \cite{gapavi09} the application of the comparison principle proceeds via 
separately estimating the gain and loss terms, the procedure which cannot be carried out in a non-cutoff case. 
Despite not using the comparison principle\footnote{Instead, we 
modify the contradiction argument from the recent work of Silvestre \cite{si14}. },
our proof of a propagation in time of exponentially decaying point-wise  
estimates carries a similarity to the idea of \cite{gapavi09}, in the sense that we too 
employ the estimates coming from the propagation of exponentially weighted 
$L^1$ norms of the solution, i.e. we ``enhance" weighed $L^1$ estimate to obtain
weighted $L^{\infty}$ estimates.

\subsection{Recent $L^\infty$ bounds for the Boltzmann equation}

Recently Silvestre \cite{si14} obtained certain regularity results 
for the Boltzmann equation in a non-cutoff case, via 
introducing at the level of the Boltzmann equation techniques 
inspired by the theory of integro-differential equations. Along the way, 
Silvestre \cite{si14} proved the following pointwise bound 
for a solution to the Boltzmann equation.

\begin{theorem}[Non-weighted pointise bounds, non-cutoff case, \cite{si14}]
\label{th-si14} 
Suppose $f(t,v)$ is a weak solution to the Boltzmann equation \eqref{eq-be}  that satisfies conditions (i) and (ii) of Theorem 3.1 with the weights $w(v,\alpha_0, p) = 1$. Then
$$ \|f(t,v)\|_{L^\infty_v} \le a + b t^{-\beta},$$
for some constants $a, b, \beta$ depending only on the initial energy, mass and entropy.
\end{theorem}
We note that the statement is written so that it can be easily compared to the statement of Theorem 3.1. In \cite{si14} the solution was referred to as a classical solution.

In this paper we generalize Theorem \ref{th-si14}, to obtain a propagation in time of  {\it weighted} $L^{\infty}$ norms of a solution. The proof builds on the known weighted $L^1$ bounds, and one of the key tools used in that direction is the Carleman representation (Lemma \ref{Carl}).

The following lemma from \cite{si14} provides an estimate that we use on the kernel $K_f$ (see \eqref{kf} for the definition of $K_f$). This lemma uses the specific structure of the angular kernel as given in \eqref{eq-b} and we will explain this more in the Appendix \ref{ap-b}. 

\begin{lemma}[Corollary 4.2, \cite{si14}]\label{si1}
For the angular kernel that satisfies \eqref{eq-b}, the weight function $K_f$ in the Carleman representation \eqref{Carl} satisfies
\begin{align}\label{lm-si1}
K_f(t,v,v') \, \approx \, \left(\int_{\{ w: w \cdot (v'-v) =0 \}} f(v+ w) \, |w|^{1+\gamma+\nu} \, dw\right) \, |v'-v|^{-d-\nu}
\end{align}
\end{lemma}

On the other hand, the following lemma from \cite{si14} provides a lower bound on the kernel $K_f$ in the Carleman representation on a distinguished set of points that lie on a certain cone. Its proof uses the representation from the above lemma.
\begin{lemma}[Lemma 7.1, \cite{si14}]\label{si2}
Suppose $f$ is a nonnegative function on $\R^d$ such that
\begin{align}\label{meh2}
  M_1 \leq &\int_{\R^d} f(v) dv \leq M_0 \nonumber\\
& \int_{\R^d} |v|^2 f(v) dv \leq E_0\\
& \int_{\R^d} f(v) \log f(v) dv \leq H_0. \nonumber
\end{align}
Then, for any $v\in\R^d$, there exists a symmetric subset $A(v)$ of the unit sphere, and there are constants $\mu, \lambda, C$ (that depend on mass, energy and entropy bounds) so that
\begin{itemize}
	\item[(i)]  $|A(v)| \geq \tfrac{\mu}{\langle v \rangle}$, where $|A(v)|$ denotes the $(d-1)$-Hausdorff measure of $A(v)$;
	\item[(ii)] For every $v'$ for which the normalized vector $\tfrac{v'-v}{|v'-v|}$ belongs to the set $A(v)$, we have
		\begin{align}\label{eq-lower Kf}
			K_f(v,v') \geq \lambda \; \langle v \rangle^{1+\gamma+\nu} \; |v'-v|^{-d-\nu},
		\end{align}
	\item[(iii)] for every $\sigma \in A(v)$, \; $|\sigma \cdot v| \leq C$.
\end{itemize}
\end{lemma}
\begin{remark} \label{cone}
Given $v$ and the corresponding subset $A(v)$ of the unit sphere determined by the above lemma, we denote by $\Sigma(v)$ 
the corresponding cone centered at $v$ of all vectors $v'$ for which the normalization $\tfrac{v'-v}{|v'-v|}$ belongs to the set $A(v)$ i.e.
$$
\Sigma(v) := \left\{  v' \in \R^d \, : \, \frac{v'-v}{|v'-v|} \in A(v) \subset \mathcal{S}^{d-1} \right\}.
$$
It is for the points $v' \in \Sigma(v)$ that the lower bound in (ii) holds.
\end{remark}
The final lemma of this section provides a lower bound of an integral over a cone $\Sigma$ determined by a vector $v$ and a subset $A$ of the unit sphere. This will be crucial in estimating the negative contribution of the collisional operator.
\begin{lemma}[Lemma 7.2, \cite{si14}]\label{si3}
Assume that the maximum of a function $g(v)$ is achieved at $v=\tilde{v}$ and is equal to $\tilde{m}$. Assume $A$ is a subset of the unit sphere and that $|A| \geq \mu >0$.  Let  $\cC$ be the cone centered at $v$ that consists of all vectors $v' \in \R^d$ for which the normalized vector $\tfrac{v'-v}{|v'-v|}$ belongs to the set $A$, i.e. $\cC := \left\{  v' \in \R^d \, : \, \frac{v'-v}{|v'-v|} \in A \right\}$. Then
\begin{align}
\int_{\cC} (\tilde{m} - g(v')) \; |\tilde{v} - v'|^{-d-\nu} \; dv' \geq \frac{c \; \tilde{m}^{1+\nu/d} \; \mu^{1+\nu/d}}{\left(\int_{\cC} |g(v')| dv'\right)^{\nu/d}}.
\end{align}
\end{lemma}
%
%


\section{Proof of Theorem \ref{thm-1}} \label{sec:proof}
To prove propagation in time of weighted $L^{\infty}_v$ norm of solutions to the Boltzmann equation, we modify the contradiction argument of Silvestre  used to prove Theorem \ref{th-si14}. 
Since we too are in the case of a non-cutoff, we cannot use the splitting of the collision operator into the ``gain" and ``loss" terms. However the standard splitting (see 
\eqref{noncut-Q12-split}) that is often used in non-cutoff cases, and which has been used by Silvestre \cite{si14} too,  is not adequate for us. We need to further refine the splitting (for details see 
\eqref{eq-splitting}) to be able to 
obtain {\it weighted} upper bounds. In particular, the appearance of the term $Q_{1,2}$ in
\eqref{eq-splitting} is new. To control that term, we
need to overcome the singularity of a non-cutoff collision operator, which we do 
thanks to oscillations present in the weight function. 
The other substantial difference with respect to \cite{si14} is that 
in our estimates we take the advantage of the known propagation of $w$-moments.


\subsection{Setting up the contradiction argument}

Let $\alpha_0 >0$ and $p\in(0,2]$ be fixed, and suppose that initial data satisfies
\begin{equation} \label{pf-id-infty}
 \left\| f_0(v) \; w(v;\alpha_0,p)\right\|_{L^\infty_v} \le C < \infty.
\end{equation}
Then for $\alpha = \alpha_0^{-}$ we have
thanks to \eqref{pf-id-infty} 
\begin{align} 
\| f_0(v) \; w(v;\alpha,p) \|_{L^1_v} 
& \leq C \int \frac{ w(v;\alpha,p) }{ w(v;\alpha_0,p) } \; dv \nonumber \\ 
& \leq C \int \frac{1}{ \langle v \rangle^{d^+} }\; dv \label{pf-infty-P3} \\
& < \infty,
\end{align}
where to obtain \eqref{pf-infty-P3} we used the property (P3). Therefore,  assumption (ii) implies that there exists $\alpha_1 < \alpha_0$ and $C_1>0$ such that 
\begin{equation} \label{pf-l1} 
\| f(t,v) \; w(v;\alpha_1,p) \|_{L^1_v} \leq C_1.
\end{equation}

It is convenient to introduce the following notation. 
For parameters $\beta, p$ and for any $t \ge 0$, let $m_{\beta, p}(t)$  denote the $w(v; \beta, p)$-weighted $L^\infty$ norm in velocity, i.e.
\begin{align}\label{def-m}
m_{\beta, p}(t) := \left\| f(t,v) \; w(v;\beta,p)\right\|_{L^\infty_v}.
\end{align}

In order  to prove the theorem, it suffices to find $\alpha_2, a, b >0$ such that 
\begin{align}\label{eq-goal}
m_{\alpha_2, p}(t) < a + b \, t^{-d/\nu}.
\end{align}

First, we show that \eqref{eq-goal} is true at $t=0$
for $\alpha_2 < \alpha_0$ and $a,b > 0$ that will be determined later in the proof. Namely, by the property (P1) that expresses monotonicity of 
$w(v;\beta, p)$ in $\beta$, we have 
\begin{equation} \label{pf-m-0-a2}  
m_{\alpha_2, p}(0) \leq m_{\alpha_0, p}(0) < \infty,
\end{equation} 
where the last inequality follows from \eqref{pf-id-infty}. 
On the other hand, $a+bt^{-d/\nu}$ blows up around $t=0$. Thus, the inequality \eqref{eq-goal} trivially holds for $t=0$, and by the continuity of $m_{\alpha_2, p}(t)$  it is satisfied on a time interval of positive measure starting at $t=0$. 

Now, assume that there exists the first time $t_0 >0$ for which the inequality \eqref{eq-goal} fails. At the time $t_0$
\begin{align}
m_{\alpha_2, p}(t_0) = a + b t_0^{-d/\nu}.
\end{align}

Since $f$ satisfies property (ii) in Theortem \ref{thm-1} for the weight $w(v; \alpha_0, p)$, it also satisfies property (ii) with the weight function $w(v; \alpha_2, p)$ since $\alpha_2<\alpha_0$ (see Remark \ref{app-rem}). Therefore, for every time $t$ the norm $L^\infty_{w(v; \alpha_2, p)}$ of $f(t,v)$, i.e. $m_{\alpha_2, p}(t)$, is attained for some velocity $v$. Let $v_0$ be such velocity corresponding to time $t_0$. In other words, 
\begin{align}\label{eq-achieved}
m_{\alpha_2, p}(t_0) = f(t,v_0) \; w(v_0; \alpha_2 ,p) = a + b t_0^{-d/\nu}.
\end{align}
Hence,
\begin{align} \nonumber
& f(t,v_0) \; w(v_0; \alpha_2 ,p) < a + b t^{-d/\nu}, \quad  \forall t < t_0,\\
& f(t_0,v_0) \; w(v_0; \alpha_2 ,p) = a + b t_0^{-d/\nu}.
\end{align}
Therefore, 
\begin{align}\label{eq-deriv}
\partial_t \left( f(t,v_0) \; w(v_0; \alpha_2 ,p) \right)_{t=t_0} \geq \partial_t \left( a + bt^{-d/\nu} \right)_{t=t_0}.
\end{align}
Combining \eqref{eq-achieved} and \eqref{eq-deriv}, we conclude the following lower bound at $(t_0, v_0)$
\begin{align}\label{eq-lower}
\partial_tf(t_0, v_0) \geq - \frac{d}{\nu} b^{-\nu/d} \; \frac{1}{ w(v_0; \alpha_2 ,p)} \; \left(m_{\alpha_2, p}(t_0) - a\right)^{1 + \frac{d}{\nu}}.
\end{align}

In the rest of the proof we look for an upper bound on $\partial_t f(t_0, v_0)$ using the Boltzmann equation \eqref{eq-be}. In particular, we estimate the collision operator $Q(f,f)(t_0, v_0)$. The upper bound that we will obtain will contradict \eqref{eq-lower} and will thus conclude our proof.

In the rest of the proof, if parameters of the weight function $w$ are not specified, they are assumed to be $\alpha_2$ and $p$.

\subsection{Splitting of the collisional operator}
When the Grad's cutoff is not assumed, it is often convenient to split the collisional integral into the following two terms, both of which are finite (\cite{de97}, \cite{al99}, \cite{vi99}, \cite{alvi02}, \cite{si14}) 
\begin{align}
& Q(f,f) = Q_1(f,f) + Q_2(f,f), \label{noncut-Q12-split}\\
& Q_1(f,f) \; = \; \int_{\R^d} \int_{\mathcal{S}^{d-1}} (f' - f) f'_* B \; d\sigma dv_* 
\; = \;  \int_{\R^d} (f' - f) K_f(v,v') \; dv',\\
&Q_2(f,f) \; = \; f(v) \; \int_{\R^d} \int_{\mathcal{S}^{d-1}} (f'_* - f_*) B \; d\sigma dv_*.
\end{align}

Since we study weighted norms, we introduce a new splitting of $Q$ tailored for the building blocks of our calculations, which are functions of the type $f w$. More precisely, 
we further split $Q_1$ into $Q_{1,1}$ and $Q_{1,2}$ according to $$Q_1 =  Q_{1,2} + Q_{1,2},$$ where
\begin{align}\label{eq-three qs}
&Q_{1,1}(f,f) \; = \; \frac{1}{w(v)} \; \int_{\R^d} \left( f' \; w' - f\; w\right) \; K_f(v,v') \; dv',\\
&Q_{1,2}(f,f) \; = \; \int_{\R^d} f' \; w' \; \left(\frac{1}{w'}  -\frac{1}{w}\right) \;  K_f(v,v') \; dv'.
\end{align}
Hence our overall decomposition of the collisional operator is
\begin{align}\label{eq-splitting}
Q(f,f) \; = \; Q_{1,1}(f,f) \; + \; Q_{1,2}(f,f) \; + \; Q_2(f,f).
\end{align}

This splitting helps us to identify the negative contribution within $Q_{1}$  at $(t_0, v_0)$. This negative contribution is coming from $Q_{1,1}(f,f)(t_0,v_0)$.
More precisely, recalling that at time $t=t_0$ the $L^\infty$ norm defining $m_{\alpha_2, p}(t_0)$ is attained at $v_0$, i.e. $m_{\alpha_2, p}(t_0) = \left\|f(t_0,v) \; w(v)\right\|_{L^\infty_v} = f(t_0, v_0) \, w(v_0)$. Therefore,
\begin{align}\label{eq-Q11}
Q_{1,1}(f,f)(t_0,v_0) \; = \; -  \frac{1}{w(v_0)} \; \int_{\R^d} \left( m_{\alpha_2, p}(t_0) - f(t_0,v') \; w(v') \right) \; K_f(v_0,v') \; dv'.
\end{align}
Since the integrand is a positive function, $Q_{1,1}(f,f)(t_0, v_0)$ is negative. However this information is not sufficient, and we proceed 
to obtain a precise upper bound on  $Q_{1,1}(f,f)(t_0, v_0)$, as well as on the other two terms. That is what we do below.

\subsection{Estimating $\boldsymbol{ Q_{1,1}}$}
As noted above, $Q_{1,1}(f,f)$ is negative at  $(t_0, v_0)$. To estimate how negative it is, we reduce the domain of integration to the cone $\cC(v_0)$ on which the lower bound \eqref{eq-lower Kf} on $K_f$ is known to hold. This cone was introduced in Lemma \ref{si2} and Remark \ref{cone}. This yields
\begin{align}\label{eq-Q11 post 7.1}
Q_{1,1}(t_0, v_0) \leq - C \;
\frac{\langle v_0 \rangle^{1+\gamma+\nu} }{w(v_0)} \; 
\int_{\cC(v_0)} \left( m_{\alpha_2, p}(t_0) - f(t_0,v') \; w(v')\right) \; |v'-v_0|^{-d-\nu} \; dv'.
\end{align}
The above integral, over the cone $\cC(v_0)$, is then estimated using Lemma \ref{si3} with $g = f w$ and its maximum value $\tilde{m} = m_{\alpha_2, p}(t_0)$. This implies
\begin{align}\label{eq-Q11 post 7.2}
Q_{1,1}(f,f)(t_0,v_0) \; \leq \; - C  \frac{ \langle v_0 \rangle^{1+\gamma+\nu}}{w(v_0)} \;  (m_{\alpha_2, p}(t_0))^{1+\nu/d} \;\frac{\left(\frac{1}{\langle v_0 \rangle}\right)^{1+\nu/d}}{\left(\int_{\cC(v_0)} f' \; w' dv'\right)^{\nu/d}}.
\end{align}
We proceed the estimate by considering the above integral in two cases, when $|v_0|\le R$ and when $|v_0|>R$, where the number $R$ is determined in the following way. Recall the statement in Lemma \ref{si2} (iii) according to which for every $\sigma \in A(v_0)$, where $A(v_0)$ is the symmetric subset of the unit sphere that determines the cone $\cC(v_0)$, we have $|\sigma \cdot v_0| \leq C$. This means that set $A(v_0)$ lies in a band of the unit sphere of width at most $C/|v_0|$ ``around the largest circle on the sphere belonging to the hyperplane that is perpendicular to" $v_0$. Hence, the larger $|v_0|$ is, the thinner the band is. Therefore, there exists number $R$ (depending on $C$), as is noted in \cite{si14}, such that
\begin{align}\label{R}
| v' | > \frac{| v_0 |}{2}, \quad \mbox{whenever} \; v'\in \cC(v_0) \; \mbox{and} \; |v_0|>R.
\end{align}
{\bf Case 1: $\boldsymbol{|v_0| \leq R}$}. It immediately follows that
\begin{align} \label{case1-1}
\frac{1}{\langle v_0 \rangle} \ge \frac{1}{\langle R \rangle },
\end{align}
and consequently
\begin{align} \label{case1-2}
\frac{1}{w(v_0)} \ge \frac{1}{w(R)}.
\end{align} 
due to the property (P1) according to which the weight $w$ is strictly positive and radially decreasing.  
In addition, since $\alpha_2 < \alpha_1$, by \eqref{pf-l1}  we have 
\begin{align} \label{case1-3}
\int_{\cC(v_0)} f' \; w' dv' \le \int_{\R^d} f' \; w' dv' & = \|f(t,v) \, w(v;\alpha_2, p)\|_{L^1_v} \nonumber\\
& \leq  \|f(t,v) \, w(v;\alpha_1, p)\|_{L^1_v} \le C_1,
\end{align}
 Applying estimates \eqref{case1-1}-\eqref{case1-3} to \eqref{eq-Q11 post 7.2} yields the following estimate on $Q_{1,1}$ whenever  $|v_0| \leq R$
\begin{align}\label{eq- Q11 small v}
Q_{1,1}(f,f)(t_0,v_0) \; & \leq \;  - \frac{C_R \; \langle v_0 \rangle^{1+\gamma+\nu} \;  (m_{\alpha_2, p}(t_0))^{1+\nu/d}}{\left(\|f(t,v) \, w(v;\alpha_1, p)\|_{L^1_v}\right)^{\nu/d}}  
\end{align}
where $C_R$ also depends on $R$.

{\bf Case 2: $\boldsymbol{|v_0| > R}$.} Now we need a more refined bound on $\int_{\cC(v_0)} f' \; w' dv'$ than the one given by \eqref{case1-3}. To find such a bound, recall from \eqref{R} that for $|v_0|>R$ and for any $v' \in \cC$
\begin{align}\label{eq-v'v0}
|v_0| \;< \; 2|v'|.
\end{align}
For $w(v_0, \alpha_3, p)$, where $\alpha_3$ will be chosen below, we have
\begin{align} 
\int_{\cC} f'& w (v'; \alpha_2, p) \; dv' \;
= \; \int_{\cC} f' \; w(v'; \alpha_2, p) \frac{w(v_0; \alpha_3, p)}{w(v_0; \alpha_3, p)} \; dv' \nonumber\\ 
& \leq \; \frac{1}{w(v_0; \alpha_3, p)} \; \int_{\cC} f' \; w(v'; \alpha_2, p) \; w(2v'; \alpha_3, p) \; dv' \label{case2-1}\\ 
& \leq \;  \frac{1}{w(v_0; \alpha_3, p)} \; \int_{\cC}f' \; w(v'; \alpha_2 + c_2 \alpha_3, p) \; dv' \label{case2-2}\\ 
& \leq \;  \frac{\|f(t,v) \, w(v;\alpha_2 + c_2 \alpha_3, p)\|_{L^1_v}}{w(v_0; \alpha_3, p)}  \nonumber \\
& \leq \;  \frac{\|f(t,v) \, w(v;\alpha_1, p)\|_{L^1_v}}{w(v_0; \alpha_3, p)} \label{case2-3a}
\end{align}
which is uniformly bounded by \eqref{pf-l1}. To obtain \eqref{case2-1} we used monotonicity  of $w$ with respect to $v$ secured by property (P1).  
To obtain \eqref{case2-2} we used the property (P2). The inequality \eqref{case2-3a} holds provided that $\alpha_3>0$ satisfies
\begin{align}\label{a3-1}
\alpha_2 + c_2 \alpha_3 < \alpha_1.
\end{align}

Now we estimate \eqref{eq-Q11 post 7.2} using  \eqref{case2-3a} 
\begin{align} \nonumber
Q_{1,1}(f,f)(t_0, v_0) \;
& \leq \; - C \; \frac{\langle v_0 \rangle^{1+\gamma+\nu}}{w(v_0; \alpha_2, p)} \;(m_{\alpha_2, p}(t_0))^{1+\nu/d} \; \frac{ \left(\frac{1}{\langle v_0 \rangle} \right)^{1+\frac{\nu}{d}}}{\left( \frac{\|f(t,v) \, w(v;\alpha_1, p)\|_{L^1_v}}{w(v_0; \alpha_3, p)}\right)^{\nu/d}}\\\nonumber 
& = \; - \frac{C \; \langle v_0 \rangle^{1+\gamma+\nu}  \; (m_{\alpha_2, p}(t_0))^{1 + \frac{\nu}{d}}}{\left( \|f(t,v) \, w(v;\alpha_1, p)\|_{L^1_v}\right)^{\nu/d}}  \; \langle v_0 \rangle^{-1- \frac{\nu}{d}}\; 
\frac{\left( w(v_0; \alpha_3, p)\right)^{\nu/d}}{w(v_0; \alpha_2, p)}  \\ \label{case2-4} 
& \le \; - \frac{C \; \langle v_0 \rangle^{1+\gamma+\nu}  \; (m_{\alpha_2, p}(t_0))^{1 + \frac{\nu}{d}}}{\left( \|f(t,v) \, w(v;\alpha_1, p)\|_{L^1_v}\right)^{\nu/d}}  \; \langle v_0 \rangle^{-1- \frac{\nu}{d}}\; \langle v_0 \rangle^2 \\ \label{case2-5}
& \le  \;  - \frac{C \; \langle v_0 \rangle^{1+\gamma+\nu}  \; (m_{\alpha_2, p}(t_0))^{1 + \frac{\nu}{d}}}{\left( \|f(t,v) \, w(v;\alpha_1, p)\|_{L^1_v}\right)^{\nu/d}},
\end{align}
where to obtain \eqref{case2-4} we use the property (P3) according to which 
$$\frac{w(v_0; \alpha_3, p)^{\nu/d}}{w(v_0; \alpha_2, p)} \geq C \langle v \rangle^2,$$ 
provided that
\begin{equation}\label{a3-2}
\frac{\alpha_3 \;\nu}{d}> \alpha_2.
\end{equation}
Now we pause for a moment to choose $\alpha_3$ to satisfy \eqref{a3-1} and \eqref{a3-2}. In particular, we choose $\alpha_3$ such that 
$$
\frac{\alpha_3 \, \nu}{d} = 2 \alpha_2,
$$
which automatically satisfies \eqref{a3-2}. Then \eqref{a3-1} implies the condition on $\alpha_2$
\begin{align}
\alpha_2 < \frac{\alpha_1}{1 + \frac{2 \, c_2 \, d}{\nu}}.
\end{align}

For such $\alpha_2$, the estimates \eqref{eq- Q11 small v} and \eqref{case2-5} imply
\begin{align}\label{q11}
Q_{1,1}(f,f)(t_0,v_0) \leq
  - \frac{C_R \; \langle v_0 \rangle^{1+\gamma+\nu} \;  (m_{\alpha_2, p}(t_0))^{1+\nu/d}}{\left(\|f(t,v) \, w(v;\alpha_1, p)\|_{L^1_v}\right)^{\nu/d}}.
\end{align}

\subsection{Estimating $Q_{1,2}$}
 Recall the definition of $Q_{1,2}$ from \eqref{eq-three qs}
\begin{align}\nonumber
Q_{1,2}(f,f) \; = \; \int_{\R^d} f' \;w' \; ( \frac{1}{w'} - \frac{1}{w}) \;  K_f(v,v') \; dv'.
\end{align}
We start by a simple observation. Since $m_{\alpha_2, p}(t)$ is defined as a supremum of $f(t,v) \; w(v)$ over velocities $v$, we have $f' w' \le m_{\alpha_2, p}$ for every $v'\in\R^d$. Therefore, 
\begin{align} \nonumber
Q_{1,2}&(f,f)(t,v) \\ \nonumber
& \leq \; m_{\alpha_2, p}(t) \; \int_{\R^d} ( \frac{1}{w'} - \frac{1}{w}) \; K_f(v,v') \; dv' \\ \label{eq-Q12 first} 
& = \; m_{\alpha_2, p}(t) \; \int_{\R^d} \left( \frac{1}{w(v+z)} - \frac{1}{w(v)} \right) \; K_f(v,v+z) \; dz.
\end{align}

Since the kernel $K_f(v,v+z)$ has a singularity at $z=0$, we estimate the above integral inside the unit ball and outside the unit ball separately, using different bounds on  $ \frac{1}{w(v+z)} - \frac{1}{w(v)}$ in those regions.

\bigskip

{\bf Outside the unit ball.} 
Since the singularity of $K_f(v, v+z)$ is at $z=0$, which is outside the considered region, a coarse bound 
$$
\left|\frac{1}{w(v+z)} -  \frac{1}{w(v)} \right| \le C,
$$ 
which follows from the property (P1). 
Applying Lemma \ref{si1}, followed by a sperical change of coordinates, yields
\begin{align}\nonumber
& \int_{|z|>1}  \left| \frac{1}{w(v+z)} - \frac{1}{w(v)}\right|  \; K_f(v,v') \; dv' \\  \nonumber
& \quad \leq \; C \; \int_{|z|>1}   \left(\int_{\{w: w \cdot z =0\}} f(v+w) |w|^{1+\gamma+\nu} \; dw \right) \; |z|^{-d - \nu} \; dz \\ \nonumber
& \quad = \; C \;  \int_1^\infty \int_{\mathcal{S}^{d-1}}  \left(\int_{\{w: w \cdot z =0\}} f(v+w) |w|^{1+\gamma+\nu} \; dw \right) \; \rho^{-d - \nu } \; \rho^{d-1} \; dS(z) \; d\rho \\ \nonumber
& \quad = \; C \; \left( - \rho^{-\nu}\right)_1^\infty \; \int_{\mathcal{S}^{d-1}}  \left(\int_{\{w: w \cdot z =0\}} f(v+w) |w|^{1+\gamma+\nu} \; dw \right) \; dS(z) \\ \label{q12-1}
& \quad = C \;\int_{\R^d} f(v+y) \; |y|^{\gamma+\nu} \; dy \\  \label{q12-2}
& \quad\leq \; C \langle v \rangle^{\gamma+\nu},
\end{align}
where to obtain \eqref{q12-1} we applied Lemma \ref{lm-change} and used the fact that
\begin{align}
 \quad \nu >0.
\end{align}
The inequality \eqref{q12-2} follows from the change of variables combined with the generation of polynomial moments and conservation of mass.

\bigskip

{\bf Inside the unit ball $\boldsymbol{|z|<1}$.} Here we need a better bound on  $\left| \frac{1}{w(v+z)} -\frac{1}{w(v)}\right|$ to compensate for the singularity of $K_f(v, v+z)$ at $z=0$.
By the mean-value theorem, we have for some $t \in [0,1]$
\begin{align}
	 \left| \frac{1}{w( v+z)} -\frac{1}{w( v)}\right| \nonumber
	& = \; \left| \nabla \left( \frac{1}{w} \right)( t v  + (1-t) (v+z)) \cdot(v+ z - v) \right| \\ \label{inside-1}
	& \leq \; C \left(  t\langle v \rangle + (1-t) \langle v+z \rangle\right)\,|z| \\ \label{inside-2}
	& \le \; \left( \langle v \rangle + |z|\right) \, |z| \\ \nonumber
	& \le 2 \langle v \rangle \; |z|,
\end{align}
where \eqref{inside-1} follows from the property (P4), while the inequality \eqref{inside-2} follows from an elementary inequality $\langle v+z \rangle \le \langle v \rangle + |z|$. Therefore, applying again Lemma \ref{si1} and spherical change of coordinates yields
\begin{align} \nonumber
& \int_{B_1} \left| \frac{1}{w(v+z)} -\frac{1}{w(v)}\right|  \; K_f(v,v') \; dv' \\  \nonumber
& \leq \; C \; \langle v \rangle \; \int_{B_1} \left(\int_{\{w: w \cdot z =0\}} f(v+w) |w|^{1+\gamma+\nu} \; dw \right) \; |z|^{-d - \nu +1} \; dz \\ \nonumber
& = \;  C \; \langle v \rangle \; \int_0^1 \int_{\mathcal{S}^{d-1}}  \left(\int_{\{w: w \cdot z =0\}} f(v+w) |w|^{1+\gamma+\nu} \; dw \right) \; \rho^{-d - \nu +1} \; \rho^{d-1} \; dS(z) \; d\rho \\ \nonumber
& = \;  C \; \langle v \rangle \; \left( \rho^{1-\nu}\right)_0^1 \; \int_{\mathcal{S}^{d-1}}  \left(\int_{\{w: w \cdot z =0\}} f(v+w) |w|^{1+\gamma+\nu} \; dw \right) \; dS(z) \\ \nonumber
& = \; C \; \langle v \rangle \; \int_{\R^d} f(v+y) \; |y|^{\gamma+\nu} \; dy \\ \nonumber
& \leq C \; \langle v \rangle \; \left( C + C \langle v \rangle^{\gamma + \nu}\right) \\ \nonumber
& \leq C \; \langle v \rangle^{1+\gamma+\nu}.
\end{align}
Note that for this calculation to work we need that 
\begin{align}
 \nu \leq 1.
\end{align}

In conclusion, combining the bounds obtained for the inside and outside the ball regions, we get
\begin{align}\label{q12}
Q_{1,2}(f,f)(t,v) \; \leq \; C \; m(t) \;  \langle v \rangle^{1+\gamma+\nu}.
\end{align}

\subsection{Estimating $Q_{2}$}
Recall that $Q_2$ is defined as
$$
Q_2(f,f) \; = \; f(v) \; \int_{\R^d} \int_{\mathcal{S}^{d-1}} (f'_* - f_*) B \; d\sigma dv_*.
$$
It is well-known, from the pioneering work on cancelation properties, by Alexandre, Desvillettes, Villani and Wennber \cite{aldeviwe00}, that the above double integral can be represented as a convolution operator. Thus, $Q_2$ takes the following simplified form
\begin{align}\label{q2}
Q_2(f,f)(t,v) \; = \; (\tilde{B}*f)(v) \; f(v)
\end{align}
where 
\begin{align}
\tilde{B}(v) = C |v|^\gamma,
\end{align}
where $\gamma$ is the potential rate from the collision kernel, and $C$ is a dimensional constant depending on the angular kernel.
Because of this simplified representation, one then has the following estimate on $Q_2$
\begin{align}\label{q2}
Q_2(f,f)(t,v) \; = \; (\tilde{B}*f)(v) \; f(v) \; \leq \; 
\begin{cases}
C \; m_{\alpha_2, p}(t) \; \langle v \rangle^\gamma, & \mbox{if} \; \gamma \geq 0\\
C \;( m_{\alpha_2, p}(t))^{1 -\frac{\gamma}{d}}, & \mbox{if} \; \gamma<0.
\end{cases}
\end{align}

\bigskip
\subsection{Conclusion} In summary, the following are the estimates \eqref{q11}, \eqref{q12}, \eqref{q2} of all three parts  \eqref{eq-three qs} of the collisional operator
\begin{align}\nonumber
Q_{1,1}(f,f)(t_0,v_0) \; & \leq \; - \frac{C_R \; \langle v_0 \rangle^{1+\gamma+\nu} \;  (m_{\alpha_2, p}(t_0))^{1+\nu/d}}{\left(\|f(t,v) \, w(v;\alpha_1, p)\|_{L^1_v}\right)^{\nu/d}},\\ \nonumber
Q_{1.2}(f,f)(t_0, v_0) \; & \leq \; C \; m_{\alpha_2, p}(t_0) \langle v_0 \rangle^{1+\gamma+\nu},\\ \nonumber
Q_2(f,f)(t_0, v_0) \; & \leq \; C \; m_{\alpha_2, p}(t_0) \;  \langle v_0 \rangle^{\gamma}.
\end{align}

Combining the three estimates yields
\begin{align} 
Q(f,f)(t_0, v_0) \; 
& \leq \; - \frac{C_R \; \langle v_0 \rangle^{1+\gamma+\nu} \;  (m_{\alpha_2, p}(t_0))^{1+\nu/d}}{\left(\|f(t,v) \, w(v;\alpha_1, p)\|_{L^1_v}\right)^{\nu/d}} +  \; C \; m_{\alpha_2, p}(t_0) \;  \langle v_0 \rangle^{1+\gamma+\nu} \nonumber \\
 \nonumber
& = \; \left(  - \frac{C_R \;  (m_{\alpha_2, p}(t_0))^{1+\nu/d}}{\left(\|f(t,v) \, w(v;\alpha_1, p)\|_{L^1_v}\right)^{\nu/d}} + C\; m_{\alpha_2, p}(t_0)\right)  \langle v_0 \rangle^{1+\gamma+\nu} \\ 
& \leq \;  - \frac{1}{2}\frac{C_R \;  (m_{\alpha_2, p}(t_0))^{1+\nu/d}}{\left(\|f(t,v) \, w(v;\alpha_1, p)\|_{L^1_v}\right)^{\nu/d}}  \;  \langle v_0 \rangle^{1+\gamma+\nu} \label{concl-1}\\
& \le  \; -\frac{1}{2}   - \frac{C_R \;  (m_{\alpha_2, p}(t_0))^{1+\nu/d}}{\left(\|f(t,v) \, w(v;\alpha_1, p)\|_{L^1_v}\right)^{\nu/d}} , \label{concl-2}
\end{align}
where the inequality \eqref{concl-1} holds provided that
\begin{align*}
\left(\frac{2C}{C_R} \left(\|f(t,v) \, w(v;\alpha_1, p)\|_{L^1_v}\right)^{\nu/d} \right)^{d/\nu} \leq m_{\alpha_2, p}(t_0) = a + b t_0^{-d/\nu},
\end{align*}
that is,
\begin{align}\label{pf-cond2}
\left(\frac{2C}{C_R} \right)^{d/\nu} \|f(t,v) \, w(v;\alpha_1, p)\|_{L^1_v} \leq a + b t_0^{-d/\nu}.
\end{align}
We choose $a$ to be 
\begin{align}
a &:= \left(\frac{2C}{C_R} \right)^{d/\nu} \|f(t,v) \, w(v;\alpha_1, p)\|_{L^1_v}\\
& \leq \left(\frac{2C}{C_R} \right)^{d/\nu} \;  C_1 .
\end{align}
For such $a$ \eqref{pf-cond2} is automatically satisfied.
Now, let us recall  \eqref{eq-lower} 
\begin{align}
Q(f,f)(t_0, v_0) = \partial_t f(t_0, v_0) \; \geq \; -\frac{d}{\nu} b^{-\nu/d} \; (m_{\alpha_2, p}(t_0) - a)^{1+\frac{\nu}{d}}.
\end{align}
Hence, if we choose $b$ so that
\begin{align}
\frac{c}{2} = \frac{d}{\nu} b^{-\nu/d},
\end{align}
we get the contradiction with the upper bound \eqref{concl-2}.


\section{Examples of weight functions and the proof of Corollary \ref{cor}}\label{sec:ex}

We provide several examples of functions that satisfy properties (P1)-(P4) and to which Theorem \ref{thm-1} can be applied, as will be proved bellow.

\noindent {\bf Example 1.}
$$w_1(v; \alpha, p) =  e^{\alpha \langle v \rangle^p}.$$

Now we proceed to check that $w_1$ indeed satisfies (P1)-(P4). It is easy to see that $w_1(v; \alpha, p)$ is strictly positive, radially increasing in $v$, and increasing in $\alpha$. Therefore it satisfies property (P1). 

Next, note that for any $\alpha_1, \alpha_2, p >0$ we have 
\begin{align*}
	w_1(v; \alpha_1, p) \; w(2v; \alpha_2, p)  
	& =  e^{\alpha_1 \langle v \rangle^p} \; e^{\alpha_2 \langle 2v \rangle^p}\\
	& \leq e^{(\alpha_1 + 2^p\alpha_2) \langle v \rangle^p}\\
	& = w_1(v; \; \alpha_1 + 2^p \alpha_2, \; p),
\end{align*}
thus $w_1$ satisfies condition (P2) as well. 

To check that condition (P3) holds, let $\delta \in [0,1]$, and let $\alpha_1, \alpha_2, p >0$ and $k\ge 0$. If $\delta \alpha_1  < \alpha_2$, then   
\begin{align*}
	\frac{w_1(v; \alpha_1, p)^\delta}{w_1(v; \alpha_2,p)} 
	& = \frac{e^{\delta \alpha_1 \langle v \rangle^p}}{e^{\alpha_2 \langle v \rangle^p}}\\
	& = \;  e^{(\delta \alpha_1 - \alpha_2) \langle v \rangle^p}\\
	& \leq C D \, \langle v \rangle^k,
\end{align*}
where $C$ is a constant that depends on parameters $k, \delta, \alpha_1, \alpha_2, p$. The last inequality holds because $\delta \alpha_1 - \alpha_2 < 0$, so the exponential $e^{(\delta \alpha_1 - \alpha_2) \langle v \rangle^p}$ decays faster than any polynomial.

Similarly, if $\delta \alpha_1  > \alpha_2$, then   
\begin{align*}
	\frac{w_1(v; \alpha_1, p)^\delta}{w_1(v; \alpha_2,p)} 
	& =   e^{(\delta \alpha_1 - \alpha_2) \langle v \rangle^p}\\
	& \ge D \, \langle v \rangle^k,
\end{align*}
where $D$ is a constant that depends on parameters $k, \delta, \alpha_1, \alpha_2, p$. The last inequality holds because $\delta \alpha_1 - \alpha_2 > 0$, so the exponential $e^{(\delta \alpha_1 - \alpha_2) \langle v \rangle^p}$ grows faster than any polynomial. In conclusion, $w_1$ satisfies condition (P3).

Finally, it is easy to check that for any $\alpha, p >0$ we have
\begin{align*}
\left| \nabla_v \left(\frac{1}{w_1(v; \alpha, p) } \right) \right| 
	& = \left| \nabla_v \left( e^{-\alpha \langle v \rangle^p} \right) \right| \\
	& \le |v| \; \left (\alpha p \; \langle v \rangle^{p-2} \; e^{-\alpha \langle v \rangle^p} \right)\\
	& \le \langle v \rangle \; \left( q + \alpha p \; \langle v \rangle^{p-2} \frac{1}{\alpha \langle v \rangle^p}\right)\\
	& \le (q+p) \langle v \rangle.
\end{align*}
Therefore $w_1$ satisfies property (P4).

\noindent{\bf Example 2.}
Second example are Mittag-Leffler functions
$$
w_2(v; \alpha, p) = \cE_{2/p} (\alpha^{2/p} \langle v \rangle^2).
$$
For simplicity we now verify that $w_2(v;\alpha, p)$, i.e. a Mittag-Leffler function, satisfies (P1)-(P4), because those functions are used in Corollary \ref{cor}.

 Recall from \eqref{eq-asym1} that 
$$
c e^{\alpha \langle v \rangle^p} \le \cE_{2/p} (\alpha^{2/p} \langle v \rangle^2) \le C e^{\alpha \langle v \rangle^p}.
$$
Using this equivalence relation and properties of classical exponential functions proved in Example 1, it is easy to check that a Mittag-Leffler function $w_2(v; \alpha, p)$ satisfies first three properties (P1)-(P3). It remains to show that it satisfies condition (P4) as well.
\begin{align*}
	\left|  \nabla_v \left(\frac{1}{\cE_{2/p} (\alpha^{2/p} \langle v \rangle^2)} \right)\right| 
	& = \left| \nabla_v \left(\sum_{k=0}^\infty \frac{\alpha^{2k/p} \; \langle v \rangle^{2k}}{\Gamma(\frac{2k}{p} +1)} \right)^{-1}\right| \\
	& \le \left(\sum_{k=1}^\infty \frac{2k \alpha^{2k/p}\; \langle v \rangle^{2k-1}}{\Gamma(\frac{2k}{p}+1)}\right) \; \left(\sum_{k=0}^\infty \frac{\alpha^{2k/p} \; \langle v \rangle^{2k}}{\Gamma(\frac{2k}{p} +1)} \right)^{-2}\\
	& \le p \left( \sum_{k=1}^\infty \frac{\alpha^{2k/p} \; \langle v \rangle^{2k-1}}{\Gamma(\frac{2k-2}{p} + 1)}\right) \left(\sum_{k=0}^\infty \frac{\alpha^{2k/p} \; \langle v \rangle^{2k}}{\Gamma(\frac{2k}{p} +1)} \right)^{-2},
\end{align*}
where in the last inequality we used  
$$ \frac{2k}{\Gamma(\frac{2k}{p} +1)} = \frac{p}{\Gamma(\frac{2k}{p})}  \le \frac{p}{\Gamma(\frac{2k-2}{p}+1)}.$$
Therefore, by simple algebraic manipulations, we get
\begin{align*}
\left|  \nabla_v \left(\frac{1}{\cE_{2/p} (\alpha^{2/p} \langle v \rangle^2)} \right)\right| 
	& \le p \; \alpha^{2/p} \; \langle v \rangle  \left( \sum_{k=1}^\infty \frac{\alpha^{(2k-2)/p} \; \langle v \rangle^{2k-2}}{\Gamma(\frac{2k-2}{p} + 1)}\right) \left(\sum_{k=0}^\infty \frac{\alpha^{2k/p} \; \langle v \rangle^{2k}}{\Gamma(\frac{2k}{p} +1)} \right)^{-2}\\
	& = p \; \alpha^{2/p} \; \langle v \rangle  \left( \sum_{k=0}^\infty \frac{\alpha^{2k/p} \; \langle v \rangle^{2k}}{\Gamma(\frac{2k}{p} + 1)}\right) \left(\sum_{k=0}^\infty \frac{\alpha^{2k/p} \; \langle v \rangle^{2k}}{\Gamma(\frac{2k}{p} +1)} \right)^{-2} \\
	& =  p \; \alpha^{2/p} \; \langle v \rangle  \left(\sum_{k=0}^\infty \frac{\alpha^{2k/p} \; \langle v \rangle^{2k}}{\Gamma(\frac{2k}{p} +1)} \right)^{-1}\\
	& \le   p \; \alpha^{2/p} \; \langle v \rangle; 
\end{align*}
hence the property (P4) holds for the Mittag-Leffler function $w_2(v, \alpha, p)$.

We are now in a position to prove Corollary \ref{cor}.

\begin{proof}[Proof of Corollary \ref{cor}]
We provide details of the proof of (a). Part (b) can be proved in an  analogous way.  First we observe that $$C e^{\alpha \langle v \rangle^p} = C w_1(v; \alpha, p),$$
where $w_1$ is the function introduced in Example 1. Therefore we know that $C e^{\alpha \langle v \rangle^p} $ satisfies (P1)-(P4). On the other hand, by  Theorem \ref{thm},  
the propagation condition \eqref{pw} of Theorem \ref{thm-1} is satisfied. The claim follows from an application of Theorem \ref{thm-1}.

\end{proof}


\appendix

\section{Some known technical results}

\begin{rem} \label{app-rem}
For any $\alpha_1, \alpha_2, p >0$ with $\alpha_2 < \alpha_1$ we have that if $f$
satisfies assumption (ii) of Theorem \ref{thm-1} with the weight $w(\cdot; \alpha_1, p)$, then it satisfies the same assumption with the weight  $w(\cdot; \alpha_2, p)$.
To prove this claim, 
suppose that for every time $t$, $\|f(t,v) \; w(v; \alpha_1, p)\|_{L^\infty_v}$ is continuous in t,  finite (not necessarily uniformly in time) and the norm is attained for some $v$. 

Thanks to the property (P3), for every $t$ we have
\begin{align}\label{rem-decay}
f(t,v) \; w(v; \alpha_2, p)
& = f(t,v) \; w(v; \alpha_1, p) \frac{ w(v; \alpha_2, p)}{ w(v; \alpha_1, p)} \nonumber\\
& \leq C f(t,v) \; w(v; \alpha_1, p) \frac{1}{\langle v \rangle^2} \nonumber\\
& \leq \frac{C}{\langle v \rangle^2} \|f(t,v) \; w(v; \alpha_1, p)\|_{L^\infty_v}. 
\end{align}
Therefore, $\|f(t,v) \; w(v; \alpha_2, p)\|_{L^\infty_v}$ is finite for every $t$.

 Next, we show that $\|f(t,v) \; w(v; \alpha_2, p)\|_{L^\infty_v}$ is continuous in $t$. This will be proved using the continuity of $\|f(t,v) \; w(v; \alpha_1, p)\|_{L^\infty_v}$  in t as well as the property  (P3). Namely, suppose, on contrary, that  $\|f(t,v) \; w(v; \alpha_2, p)\|_{L^\infty_v}$ is continuous at some $t_0$. Then: 
$\exists \varepsilon,  \,\, \forall n\in\N, \,\, \exists t_n$ such that 
\begin{align*}
|t_0-t_n| <\frac{1}{n} \quad \mbox{and} \quad
\left|  \|f(t_0,v) w(v,\alpha_2,p)\|_{L^\infty_v}  -    \|f(t_n,v) w(v,\alpha_2,p)\|_{L^\infty_v} \right| >\varepsilon.
\end{align*}
Since, 
\begin{align*}
& \left| \|f(t_0,v) \; w(v; \alpha_2, p)\|_{L^\infty_v} - \|f(t_n,v) \; w(v; \alpha_2, p)\|_{L^\infty_v} \right| \\
& \qquad \leq \|(f(t_0,v)-f(t_n,v)) \; w(v; \alpha_2, p)\|_{L^\infty_v},
\end{align*}
we have that $\exists \varepsilon,  \,\, \forall n\in\N, \,\, \exists t_n$ such that 
\begin{align*}
|t_0-t_n| <\frac{1}{n} \quad \mbox{and} \quad
 \|(f(t_0,v)-f(t_n,v)) \; w(v; \alpha_2, p)\|_{L^\infty_v} >\varepsilon.
\end{align*}
Moreover, by the definition of supremum, we now have that $\exists \varepsilon,  \,\, \forall n\in\N, \,\, \exists t_n, \,\, \exists v_n$ such that 
\begin{align} \label{ap-key}
|t_0-t_n| <\frac{1}{n} \quad \mbox{and} \quad
\left|(f(t_0,v_n)-f(t_n,v_n)) \; w(v_n; \alpha_2, p)\right| >\frac{\varepsilon}{2}.
\end{align}
Before we proceed, we show that such $v_n$ have to lie in a ball of finite and fixed radius. Namely, if we denote 
$$m_1(t) := \|f(t,v) w(v;\alpha_1,p)\|_{L^\infty_v},$$ then we have $f(t,v) \leq m_1(t) (w(v; \alpha_1, p))^{-1} $, and so from \eqref{ap-key}, we have
\begin{align*}
 \frac{\varepsilon}{2} & < \left( m_1(t_0) + m_1(t_n)\right)(w(v_n; \alpha_1, p))^{-1} w(v_n; \alpha_2, p)\\
& \leq C  \left( m_1(t_0) + m_1(t_n)\right) \frac{1}{\langle v_n \rangle^2},
\end{align*}
where in the last inequality we used property (P3) and therefore the constant $C$ depends only on $\alpha_1, \alpha_2$ and $p$. From here we conclude, due to the continuity of function $m_1(t)$ and form the fact that $|t_0-t_n|<\frac{1}{n}$, that for all $n \geq N_0$ (for some fixed $N_0$) we have that
\begin{align*}
\langle v_n \rangle^2 \leq \frac{2C (2m_1(t_0)+1)}{\varepsilon}.
\end{align*}
Thus, indeed,  for $n\geq N_0$, $v_n$ lie in a ball of a radius that depend only on the fixed quantities $\varepsilon, C (\alpha_1, \alpha_2, p), m_1(t_0)$. 

Now, going back to \eqref{ap-key}, we get a contradiction. Namely, $w(v_n; \alpha_2, p)$ is bounded in $n$ (since $v_n$ are bounded) and 
$f(t_0, v_n) - f(t_n,v_n)$ converges to zero as $n \rightarrow \infty$ since $f$ is countinuius in $(t,v)$, $|t_0-t_n| <\frac{1}{n}$ and $v_n$ lie in a fixed ball (and thus converge to some point). Therefore, 
$|(f(t_0,v_n)-f(t_n,v_n)) \; w(v_n; \alpha_2, p)| \rightarrow 0$ as $n \rightarrow \infty$, so \eqref{ap-key} cannot hold. So our assuption that $\|f(t,v) \; w(v; \alpha_2, p)\|_{L^\infty_v}$ is not continuous in $t$ was wrong.

In order to see that this supremum is achieved, fix an arbitrary time $t_0$, suppose that 
$$\|f(t_0,v) \; w(v; \alpha_2, p)\|_{L^\infty_v} = C_0$$ 
and suppose on contrary that this supremum is not attained. That is, suppose that there is a sequence $\{v_n\}_n$ so that 
\begin{align*}
&\|f(t_0,v_n) \; w(v_n; \alpha_2, p)\|_{L^\infty_v} <C_0\\
&\|f(t_0,v_n) \; w(v_n; \alpha_2, p)\|_{L^\infty_v} \rightarrow C_0, \quad \mbox{as} \;\; n \rightarrow \infty.
\end{align*}
Velocities $v_n$ cannot be inside of a ball $B_R$ of finite radius $R$, because then they would converge to some $v_*\in B_R$ and at that point we would have that  $\|f(t_0,v_*) \; w(v_*; \alpha_2, p)\|_{L^\infty_v} = C_0$, which would contradict the assumption that the supremum is not achieved. Hence, there exists a subsequence, which we still call $v_n$, so that $|v_n| \rightarrow \infty$ and
\begin{align*}
C_0/2 < \|f(t_0,v_n) \; w(v_n; \alpha_2, p)\|_{L^\infty_v} <C_0.
\end{align*}
This contradicts the decay in \eqref{rem-decay} as  $f(t_0,v_n) \; w(v_n; \alpha_2, p) \leq C \langle v_n \rangle^{-2} \rightarrow 0$, so the lower bound could not hold. This concludes the proof of the remark.
\end{rem}

We state a classical change of variable result. 

\begin{lemma} \label{lm-change}
Suppose $g$ is any non-negative function. Then
\begin{align}
\int_{\mathcal{S}^{d-1}} \int_{\{\omega: \omega \cdot \sigma = 0 \}} 
g(\omega) \; d\omega \; dS(\sigma)
\; = \; c_d \; \int_{\R^d} g(y) \; \frac{dy}{|y|}.
\end{align}
\end{lemma}

\section{Proofs of lemmas from Section 4.3} \label{ap-c}

\begin{wrapfigure}{r}{0.35 \textwidth}
\setlength{\unitlength}{1cm}
\begin{tikzpicture}[scale=0.45]
\centering
\begin{scope}
   	\draw (0,0) circle (4);
	\draw [->] (-4,0)-- (4,0);
	\node at (-4.5, -0.5) {$v_*$};
	
	\node at (4.2, -0.5) {$v$};
	\node at (3.5, 2.8) {$ v'$};

	\draw [->] (-3,-2.6) -- (3,2.6);
	\node at (-3.3, -3.3)  {$v'_*$};

	\draw[->] [dashed] (3, 2.6)--(-4,0);
	\node at (-.50, 1.9) {w};
	
	\draw (-1, 0) arc (0:20:3);
	\node at (-1.8, 0.3) {$\theta/2$};
	
	\draw (1.5, 0) arc (0:40:1.5);
	\node at (1, 0.3) {$\theta$};

\end{scope}
\centering
\end{tikzpicture}
\vspace{-0.9in}
\end{wrapfigure}
In this section we provide proofs of lemmas from Subsection 4.3. While these lemmas were established by Silvestre in \cite{si14}, we include their proofs here for completeness purposes.

\begin{proof}[Proof of Lemma \ref{si1}]
By Lemma \ref{car}, the kernel $K_f$ can be represented as
\begin{align}\label{ap-Kf}
K_f(v,v') \; = \; \frac{2^{d-1}}{|v'-v|} \; \int_{\{w: w\cdot (v'-v) = 0\}} f(v+w) \; r^{-d+2+\gamma} \;  b(\cos\theta) dw.
\end{align}

The following part of the integrand: $r^{-d+2+\gamma} \;  b(\cos\theta)$ is estimated by considering two cases - when $\cos\theta>0$ and when $\cos\theta<0$. But first, recall the notation $r=|v-v_*|$ and $w = v_* - v'$. Thus,
\begin{align}\label{recall}
|w| =  r \cos \frac{\theta}{2}, \qquad 
|v'-v|  = r \left|\sin \frac{\theta}{2}\right|.
\end{align}

{\it Case 1:} $\cos \theta>0.$ In this case $\theta/2 \in (-\pi/4, \pi/4)$, which via \eqref{recall} implies $\frac{\sqrt{2}}{2} r\leq |w| \leq r$. In other words, 
\begin{align}\label{ap-case1-1}
|w|\approx r.
\end{align}
On the other hand, by \eqref{eq-b} 
\begin{align*}
b(\cos\theta) 
&\approx (\sin \theta)^{-d+1-\nu} \\
&= \left(2 \sin\frac{\theta}{2} \cos\frac{\theta}{2} \right)^{-d+1-\nu}\\
&\approx \left(\sin \frac{\theta}{2}\right)^{-d+1-\nu},
\end{align*}
where to obtain the last equivalence we used \eqref{recall} and \eqref{ap-case1-1}. Therefore,
\begin{align}\label{ap-case1}
r^{-d+2+\gamma} \;  b(\cos\theta) 
& \approx r^{-d+2+\gamma} \left( \frac{|v'-v|}{r}\right)^{-d+1-\nu} \nonumber\\
& = r^{1+\nu +\gamma} |v'-v|^{-d+1-\nu}\nonumber\\
& \approx |w|^{1+\nu +\gamma} |v'-v|^{-d+1-\nu}.
\end{align}

{\it Case 2:} $\cos \theta<0.$ In this case $\theta/2 \in (\pi/4, \pi/2) \cup (-\pi/2, -\pi/4)$, which via \eqref{recall} implies $\frac{\sqrt{2}}{2} r \leq |v'-v| \leq r$. In other words, 
\begin{align}\label{ap-case2-1}
|v'-v| \approx r.
\end{align}
On the other hand, by \eqref{eq-b}, 
$$b(\cos\theta) 
\approx (\sin\theta)^{1+\gamma+\nu} 
\approx \left(\cos\frac{\theta}{2}\right)^{1+\gamma+\nu},$$ 
where the last equivalence follows from \eqref{recall} and \eqref{ap-case2-1}. Therefore,
\begin{align}\label{ap-case2}
r^{-d+2+\gamma} \;  b(\cos\theta) 
& \approx r^{-d+2+\gamma} \left( \frac{|w|}{r}\right)^{1+\gamma+\nu} \nonumber\\
& = r^{-d+1-\nu} |w|^{1+\gamma+\nu}\nonumber\\
& \approx |v'-v|^{-d+1-\nu} |w|^{1+\gamma+\nu}.
\end{align}

In both cases, $r^{-d+2+\gamma} \;  b(\cos\theta)$ is approximated in the same way. Applying estimates \eqref{ap-case1} and \eqref{ap-case2} to \eqref{ap-Kf}, we have
\begin{align}
K_f(v,v') \; \approx \; \left(\int_{\{w: w\cdot (v'-v) = 0\}} f(v+w)|w|^{1+\gamma+\nu} dw\right) |v'-v|^{-d-\nu},
\end{align}
which completes the proof of the lemma.
\end{proof}
\begin{remark}
We observe that in the proof of Lemma 4.4 the format  \eqref{eq-b} of $\tilde{b}$ was used to obtain \eqref{ap-case2}. 
\end{remark}
 In order to prove Lemma \ref{si2}, we need the following result.
 
\begin{lemma}\cite[Lemma 4.6]{si14}\label{L46}
Suppose $f$ is a non-negative function defined on $\R^d$ which satisfies
\begin{align}\label{meh1}
& \int_{\R^d} f(v) dv \geq M_1 \nonumber\\
& \int_{\R^d} |v|^2 f(v) dv \leq E_0\\
& \int_{\R^d} f(v) \log f(v) dv \leq H_0 \nonumber
\end{align}
for some positive constants $M_1, E_0,  H_0$. Then there exist positive constants $r, l, m$ (depending only on $M_1, E_0$ and $H_0$) such that
\begin{align}
|\{ v: f(v)>l\} \cap B_r|\geq m.
\end{align}
\end{lemma}
\begin{proof}
First note that 
\begin{align*}
E_0 & \geq \int_{\R^d} |v|^2 f(v) dv \geq r^2 \int_{|v|>r} f(v) dv.
\end{align*}
Therefore, if we choose $r>0$ so that $E_0/r^2 < M_1/2$, then
\begin{align*}
\int_{|v|>r} f(v) dv \leq \frac{E_0}{r^2} < \frac{M_1}{2}.
\end{align*}
Consequently, using the first condition in \eqref{meh1}, we have
\begin{align}\label{eq-L46-Br}
\int_{B_r} f(v) dv \geq \frac{M_1}{2}.
\end{align}
Next, we choose $l>0$ small enough so that $l|B_r| < M_1/8 < M_1/4$. Then
\begin{align}
\int_{B_r \cap \{v: f(v)>l\}} f(v) dv
& = \int_{B_r} f(v) dv - \int_{B_r \cap \{v: f(v)\leq l\}} f(v) dv \nonumber\\
& \geq \frac{M_1}{2} - l |B_r \cap \{v: f(v)\leq l\}| \label{eq-L46-Br1}\\
& \geq  \frac{M_1}{2} - l |B_r| \nonumber\\
&\geq  \frac{M_1}{2} -  \frac{M_1}{4} =  \frac{M_1}{4}. \label{eq-L46-Brl},
\end{align}
where to obtain \eqref{eq-L46-Br1} we used \eqref{eq-L46-Br}, and to obtain \eqref{eq-L46-Brl} we relied on the way the constant $l$ was chosen.

To prove that the set $B_r \cap \{v: f(v)>l\}$ has measure that can be bounded from below in terms of $M_1, E_0, H_0$, we use the entropy. Namely, define $T>0$ so that 
\begin{align}\label{eq-L46-T}
T |B_r \cap \{v: f(v)\leq l\}| = \frac{M_1}{8}.
\end{align}
Then, since $l<T$,  we have that 
\begin{align}
\int_{B_r \cap \{v: f(v)> T\}} f(v) dv
& = \int_{B_r \cap \{v: f(v)> l\}} f(v) dv - \int_{B_r \cap \{v: T \geq f(v)> l\}} f(v) dv \nonumber \\
& \geq \frac{M_1}{4} - T |B_r \cap \{v: T\geq f(v)> l\}| \nonumber \\
& \geq \frac{M_1}{4} - T |B_r \cap \{v: f(v)> l\}| \nonumber \\
& = \frac{M_1}{4} - \frac{M_1}{8} = \frac{M_1}{8}, \label{eq-L46-BrT}
\end{align}
where the last line is obtained via \eqref{eq-L46-T}. Now, by the entropy assumption in \eqref{meh1}, we have
\begin{align*}
H_0 & \geq \int_{B_r \cap \{v: f(v)> T\}} f(v) \log f(v) dv \\
& \geq \log T \int_{B_r \cap \{v: f(v)> T\}} f(v) dv \\
& \geq \frac{M_1 \log T}{8}.
\end{align*}

Therefore, $T \leq e^{8H_0/M_1}$, and so by $\eqref{eq-L46-T}$ we obtained the desired lower bound 
\begin{align*}
|B_r \cap \{v: f(v)> l\}| = \frac{M_1}{8T} \geq \frac{M_1}{8}e^{-8H_0/M_1}.
\end{align*}

\end{proof}

We are now in a position to prove Lemma \ref{si2}.

\begin{proof}[Proof of Lemma \ref{si2}]

Since $f$ is assumed to satisfy \eqref{meh2}, Lemma \ref{L46} may be applied. Therefore, there exist $l, m, r>0$ such that
$$ |B_r \cap \{v: f(v)> l\}| \geq m. $$

Let us denote such set by $S$, that is, let
\begin{align}\label{S}
S:= B_r \cap \{v: f(v)> l\}.
\end{align}
Then immediately, we have
\begin{align}\label{S>m}
|S| \geq m.
\end{align}

The set $S$ is easily seen to satisfy $f(v) \geq l \chi_S(v)$. Hence,
\begin{align} \label{L48-1}
K_f(v,v') &\geq l K_{\chi_S}(v,v') \nonumber \\
&\geq c l \left( \int_{ \{ w: w\cdot(v'-v) =0\} } \chi_S(v+w) |w|^{1+\gamma+\nu} dw\right) |v'-v|^{-d-\nu},
\end{align}
where the last inequality follows from \eqref{lm-si1}.

The proof proceeds by considering the case of small velocities and large velocities  $|v|$ separately. Fix some $R>>r$. Consider the following two cases:

{\it Case 1:} $|v|\le R$. In this case, the vectors $w$ for which $\chi_S(v+w) \neq 0$ satisfy 
$$ v+w \in S \subset B_r.$$
Thus,
$$ |w| \leq r + |v| \leq r+R.$$
This implies the following uniform bound
\begin{align}\label{L48-2}
\Pi_v(\sigma) := \int_{ \{ w: w\cdot(v'-v) =0\} } \chi_S(v+w) |w|^{1+\gamma+\nu} dw 
\leq (r+R)^{1+\gamma+\nu} |Br|.
\end{align}
On the other hand, integrating the hyperplane integral $\Pi(v)$ in over the unit sphere, and applying the change of variables Lemma \ref{lm-change} yields
\begin{align}\label{L48-3}
\int_{\sigma \in \partial B_1} \Pi_v(\sigma)  d\sigma 
= \int_{\R^d} \chi_S(z) |z-v|^{\gamma+\nu} dz \geq \kappa >0,
\end{align}
where the last two inequalities are a consequence of \eqref{S>m} by which the set $S$ has positive measure. 

The uniform bound \eqref{L48-2} implies that the integrand $\Pi_v(\sigma)$ in \eqref{L48-3} cannot concentrate mass in sets of measure zero. In other words, it is not possible that $|\{\sigma: \Pi_v(\sigma) >\lambda \}| =0 $ for all $\lambda>0$. Therefore, there exists $\lambda_1>0$ such that $\Pi_v(\sigma) \leq \lambda_1$ on a set of positive measure. That is, there exists a subset $A \subset \partial B_1$ of positive measure such that 
$$ 
\int_{\{w: w\cdot \sigma =0\} } \chi_S(v+w) |w|^{1+\gamma+\nu} dw \geq \lambda_1, \qquad \mbox{if} \;\;\; \sigma \in A.
$$
This set $A$ is symmetric since $\sigma$ and $-\sigma$ define the same hyperplane $\{w: w\cdot \sigma =0\}$. In addition, since it is of positive measure, $|A|>\mu \geq \frac{\mu}{\langle v \rangle}$, for some $\mu >0$. This proves part (i) of Lemma 4.5.

Applying the last inequality to \eqref{L48-1} implies the following lower bounds for any $v'$ for which $\frac{v'-v}{|v'-v|} \in A$:
\begin{align*}
K_f(v',v) \geq C |v'-v|^{-d-\nu} \geq \tilde{C} \langle v \rangle^{1+\gamma+\nu} |v'-v|^{-d-\nu},
\end{align*}
where the last inequality exploits the fact that $|v| \leq R$. This proves part (ii) of the statements of Lemma 4.5.

Finally, the last statement of the Lemma, statement (iii) is trivial in this case, since
$$
|\sigma \cdot v| \leq |v| \leq R \leq \frac{R^2}{|v|}.
$$

{\it Case 2:} $|v|>R$. By the choice of $R$, $|v|>r$. Therefore, since $S \subset B_r$, we have that $v \notin S$, i.e. $0 \notin S-v$.

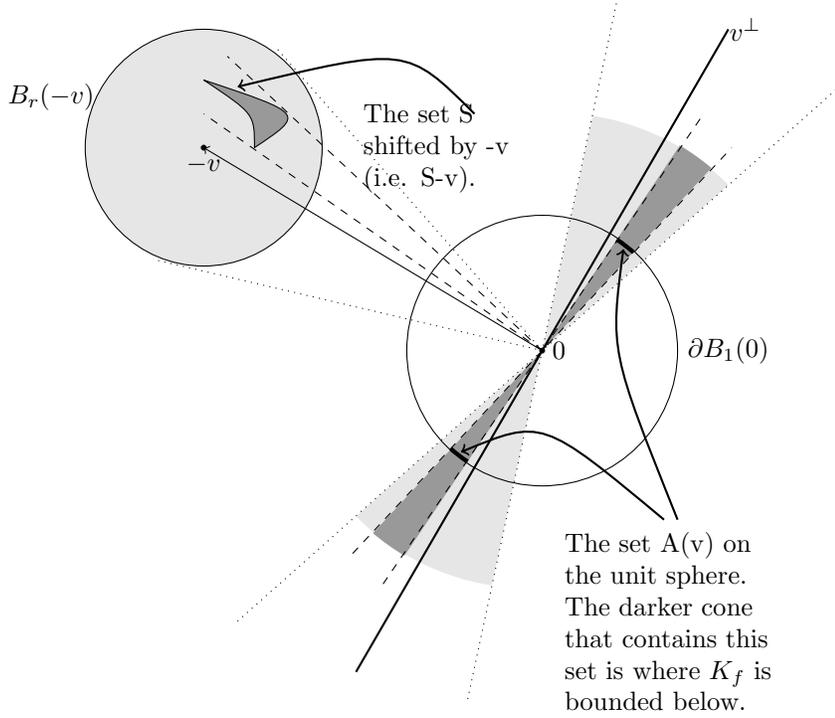
\begin{tikzpicture}[scale=0.45, baseline=(current bounding box.north)]
\centering
\begin{scope}
	\fill[fill={gray!20}] (0,0) -- (5.5,4.9) arc (45:81:7.28) -- (0,0);
	\fill[fill={gray!20}] (0,0) -- (-5.5,-4.9) arc (225:261:7.28) -- (0,0);
	
	\fill[fill={gray!80}] (0,0) -- (5,5.4) arc (50:59:7.3) -- (0,0);
	\fill[fill={gray!80}] (0,0) -- (-5,-5.4) arc (230:239:7.3) -- (0,0);

	\filldraw[fill = {gray!20}] (-10,6) circle (3.5);
	\node at (-14.5, 7.5) {$B_r (- v)$};

   	\draw (0,0) circle (4);
	\filldraw (0,0) circle (2pt);
	\node at (0.5, 0) {0};
	\node at (5.5, 0) {$\partial B_1(0)$};

	\draw [->] (0,0) -- (-10,6);
	\filldraw (-10, 6) circle (2pt);
	\node at (-10, 5.5) {$-v$};

	\draw [dotted] (0,0) -- (-7.9, 9);
	\draw [dotted] (0,0) -- (-11.5, 2.7);
	\draw [dotted] (-9,-8) -- (9,8);
	\draw [dotted] (-2.2, -10.3) -- (2.2,10.3);

	\filldraw [fill={gray!80}] (-8.5,6) .. controls (-8.5,7) .. (-10,8)
						(-10,8) .. controls (-7,7) .. (-8.5,6);

	\draw [dashed] (0,0) -- (-9.3, 8.7);
	\draw [dashed] (0,0) -- (-10,7);
	\draw [dashed] (-5.6, -6) -- (5.6, 6);
	\draw [dashed] (-4.7,-6.9) -- (4.7, 6.9);
	
	\draw [ultra thick] (2.7, 2.9) arc (48:57:4);
	\draw [ultra thick] (-2.7, -2.9) arc (228:237:4);

	\draw [->, thick] (4,-5) .. controls (2,0) .. (2.4, 3);
	\draw [->, thick] (3.6, -5) .. controls (0, -2) .. (-2.4, -3);
	\node at (4, -8) [text width = 3cm] {The set A(v) on the unit sphere. The darker cone that contains this set is where $K_f$ is bounded below.};

	\draw[->, thick] (-2, 7) .. controls (-4, 9) .. (-9, 7.8);
	\node at (-2.5, 6) [text width = 2.5cm] {The set S shifted by -v (i.e. S-v).};
	
	\draw [thick] (-5.5, -9.5) -- (5.5,9.5);
	\node at (6, 9.5) {$v^\perp$};

\end{scope}
\centering
\end{tikzpicture}

To obtain a lower bound on \eqref{L48-1}, we need to find a set of directions $\sigma$ of the vector $v'-v$ for which the set of $w$'s that are orthogonal to $\sigma$ and satisfy $\chi_S(v+w) \neq 0$ is of positive measure.

Let $\delta >0$, to be chosen later. Define
\begin{align}\label{A(v)}
A(v) :=  \left\{ \sigma \in \partial B_1: | \{\omega: \omega \cdot \sigma =0 \} \cap \{S-v \} | > \delta \right\}.
\end{align}

Then, we claim that
\begin{align} \label{lower A}
 |A(v)| > \frac{c \mu}{|v|},
 \end{align}
which would prove part (i) of the lemma. To prove \eqref{lower A}, recall that by \eqref{S>m} and by Lemma \ref{lm-change} we have
\begin{align}\label{|S|}
m \leq |S| &= |S-v| = \int_{\R^d} \chi_{S-v}(y) \;|y| \; \frac{dy}{|y|} \nonumber\\
& = \int_{\partial B_1} \int_{\{\omega: \omega \cdot \sigma =0 \}} \chi_{S-v}(w) \; |w| \;dw \; d \sigma \nonumber\\
& =  \int_{\partial B_1} \int_{\{\omega: \omega \cdot \sigma =0 \}} \chi_{S}(v+w) \; |w| \;dw \; d \sigma \nonumber\\
& =  \int_{D} \int_{\{\omega: \omega \cdot \sigma =0 \}} \chi_{S}(v+w) \; |w| \;dw \; d \sigma,
\end{align}
where $D$ is the portion of the unit sphere $\partial B_1(0)$ around the direction $v^\perp$ which contains directions $\sigma$ so that $\sigma^\perp$ intersects the ball $B_r(-v)$. In the graph above, $D$ is the portion of $\partial B_1(0)$ inside the light grey cone. The notation we use here is the following: if $x$ is a vector in $\R^d$, then $x^\perp$ is the hyperplane passing through the origin which is orthogonal to $x$. Using similarity of triangles, it can be concluded that 
\begin{align} \label{band}
|D| \leq \frac{C}{|v|}.
\end{align}
Next, we proceed with the estimate \eqref{band} by splitting the integral in cases $\sigma \in A(v)$ and $\sigma \in D\setminus A(v)$, and using that for $v+w \in S$ we have as before $|w|\leq r+|v|$:
\begin{align}
m & \leq \int_{A(v)} \int_{\{\omega: \omega \cdot \sigma =0 \}} \chi_{S}(v+w) \; |w| \;dw \; d \sigma 
+ \int_{D\setminus A(v)} \int_{\{\omega: \omega \cdot \sigma =0 \}} \chi_{S}(v+w) \; |w| \;dw \; d \sigma \nonumber  \nonumber\\
& \leq c_d  \;(r+|v|)\;|A(v)|  \;+ \;  c_d \; \delta \; \;(r+|v|) \; |D\setminus A(v)| \nonumber\\
& \leq 2 \; c_d \;|v| \; \left(A(v) + \delta \frac{C}{|v|} \right)  \nonumber
\end{align}
where $\delta$ comes from \eqref{A(v)}. Therefore,
\begin{align}
|A(v)| \geq \frac{m}{2\;  c_d \; |v|} - \frac{\delta C}{|v|}.
\end{align}
By choosing $\delta \leq \frac{m}{4\; c_d}$ and $\mu \leq m$, we have
$$
|A(v)| \geq \frac{\mu}{4 c_d |v|},
$$
which indeed proves part (i) of the statement of the Lemma 4.5.

To prove part (ii), note that since $|v|$ is large, $|w| \sim |v|$ whenever $v+w \in S$. Therefore, by \eqref{L48-1} we have that for any $v'$ for which $\frac{v'-v}{|v'-v|} \in A$
\begin{align*}
K_f(v',v) &\geq C \langle v \rangle^{1+\gamma+\nu} |v'-v|^{-d-\nu}
\int_{ \{ w: w\cdot(v'-v) =0\} } \chi_S(v+w)  dw \\
&\geq \delta C \langle v \rangle^{1+\gamma+\nu} |v'-v|^{-d-\nu}
\end{align*}
which proves part (ii) of the lemma. 

Finally, part (iii) follows from the fact that the band $D$ was of width at most $\frac{C}{|v|}$. Therefore for any $\sigma \in A(v) \subset D$, we have
$$
|\sigma \cdot v| \leq \cos(\sigma, v) |; |v| \leq \frac{C}{|v|} |v| \leq C.
$$

\end{proof}

Finally, we provide the proof of Lemma \ref{si3}.

\begin{proof}[Proof of Lemma \ref{si3}]
Let 
$$
G:= \{ v' \in \cC: f(v') \leq \frac{m}{2}\}.
$$
Then one can easily observe the following upper bound on the measure of its complement:
\begin{align}\label{upper}
|\cC \setminus G| & = \{ v' \in \cC: f(v') > \frac{m}{2}\} \nonumber \\
& \leq \frac{2}{m} \int_{\cC} |f(v')| dv'.
\end{align}

Since for every $v' \in G$ we have $(m-f(v'))\geq m/2$, we conclude that
\begin{align}\label{L47-G}
\int_{\cC} (m-f(v')) |v'-v|^{-d-\nu} dv' \geq \frac{m}{2} \int_G |v'-v|^{-d-\nu} dv'.
\end{align}
To find the lower bound on the last integral, we remark two properties. One, the complement of its domain has an upper bound \eqref{upper}. Two, the values of $|v'-v|^{-d-\nu}$ are smaller the further away $v'$ is from the center $v$ of the cone $\cC$. Therefore, the smallest possible value of the integral on the right-hand side of \eqref{L47-G} is achieved when set $G$ is as far away from the center of the cone $\cC$, and when its complement has measure $\frac{2}{m} \int_{\cC} |f(v')| dv'$ from \eqref{upper}. In other words, the lowest value is for $G= \cC \setminus B_r$, where $r$ is chosen so that
\begin{align}
|\cC\cap B_r| = \frac{2}{m} \int_{\cC} |f(v')| dv'.
\end{align}
Such $r$ then satisfies
\begin{align}
r = \left(\frac{2d}{|A|m} \int_{\cC} |f(v')| dv' \right)^{1/d}
\end{align}
Continuing the estimate \eqref{L47-G}, we now have
\begin{align}
\int_{\cC} (m-f(v')) |v'-v|^{-d-\nu} dv' 
& \geq \frac{m}{2} \int_{\cC \setminus B_r} |v'-v|^{-d-\nu} dv' \nonumber \\
& = \frac{m}{2} \int_r^\infty \int_A d\sigma s^{-d-\nu} s^{d-1} ds \nonumber \\
& = \frac{|A| m}{2} \frac{r^{-\nu}}{\nu}\\
& = c_{\nu,d}\frac{m^{1+\nu/d} |A|^{1+\nu/d}}{\left(\int_{\cC} |f(v')| dv' \right)^{\nu/d}}.
\end{align}

\end{proof}

\section{A note on the angular kernel} \label{ap-b}
In this section, we explain why was the angular kernel modified (without changing the value of collision operator $Q(f,f)$) from \eqref{b-pw} to \eqref{eq-b}. Recall the two formulas in question. First one \eqref{b-pw} was inspired by the inverse power law 
\begin{equation}
b(\cos{\theta}) \approx  (\sin{\theta})^{-(d-1) - \nu},  \quad \mbox{with} \;\nu \in (0,2).
\end{equation}
But, then it was modified on half of its domain to the following form also written in \eqref{eq-b} 
\begin{align}
b(\theta) \approx 
\begin{cases}
\left(\sin{\theta}\right)^{-d+1-\nu}, & \mbox{if} \; \cos{\theta} \geq 0 \\
\left(\sin{\theta}\right)^{1+\gamma+\nu}, & \mbox{if} \; \cos{\theta} < 0.
\end{cases}
\end{align}
This modification is made so that the kernel $K_f$ in the Carleman representation coincides for the above cases, 
$\cos{\theta} \geq 0$ and $\cos{\theta} < 0$. To better understand this choice, consider setting "arbitrary" powers
\begin{align}
b(\theta) \approx 
\begin{cases}
\left(\sin{\theta}\right)^{\tau_+}, & \mbox{if} \; \cos{\theta} \geq 0 \\
\left(\sin{\theta}\right)^{\tau_-}, & \mbox{if} \; \cos{\theta} < 0,
\end{cases}
\end{align}
and let us see what is the behavior of the kernel $K_f$ in the Carleman representation under this assumption.
Recall, the Carleman representation from \eqref{Carl} 
\begin{align}\label{Carl}
\int_{\R^d} \int_{\mathcal{S}^{d-1}} H(v,v') \; f(v'_*) \;  B(r, \theta) \; d\sigma dv_* 
\; = \;  \int_{\R^d} H(v,v') \; K _f(v,v') \; dv',
\end{align}
where the kernel $K_f(v,v')$ is given by
\begin{align}\label{kf}
K_f(v,v') \; = \; \frac{2^{d-1}}{|v'-v|} \; \int_{\{w: w\cdot (v'-v) = 0\}} f(v+w) \; B(r,\theta) \; r^{-d+2} \; dw.
\end{align}
In the new set of variables $(v,v',w)$, we have
\begin{align*}
r = \sqrt{|v'-v|^2 + |w|^2}, \; \; \cos{\frac{\theta}{2}} = \frac{|w|}{r}, \\
v'_* = v + w, \; \; v_* = v' +w.
\end{align*}
The key idea is to notice that
\begin{align}
\begin{cases}
|w| \approx r, & \mbox{if} \; \cos\theta \geq 0 \\
|v'-v| \approx r, & \mbox{if} \; \cos\theta <0.
\end{cases}
\end{align}

Therefore, the collisional kernel $B(r,\theta)$ behaves differently in terms of $|w|$ and $|v'-v|$ when cosine is positive and when cosine is negative. To be more precise,
\begin{itemize}
\item {\bf When $\theta$ is such that $\boldsymbol\cos\theta \ge 0$,} then
\begin{align}
B(r,\theta) \;  r^{-d+2}\; 
&=\;  r^{-d+2+\gamma} \; (\sin\theta)^{\tau_+} \\ \nonumber
& =  r^{-d+2+\gamma} \; \frac{|v'-v|^{\tau_+}}{r^{\tau_+}} \; \frac{|w|^{\tau_+}}{r^{\tau_+}} \\ \nonumber
& \approx  r^{-d+2+\gamma} \; \frac{|v'-v|^{\tau_+}}{r^{\tau_+}} \\ \nonumber
& \approx |w|^{-d+2+\gamma-\tau_+} \; |v'-v|^{\tau_+}.
\end{align}

\item {\bf When $\theta$ is such that $\boldsymbol\cos\theta <0$,} then
\begin{align}
B(r,\theta) \;  r^{-d+2}\; 
&=\;  r^{-d+2+\gamma} \; (\sin\theta)^{\tau_-} \\ \nonumber
& =  r^{-d+2+\gamma} \; \frac{|v'-v|^{\tau_-}}{r^{\tau_-}} \; \frac{|w|^{\tau_-}}{r^{\tau_-}} \\ \nonumber
& \approx  r^{-d+2+\gamma} \; \frac{|w|^{\tau_-}}{r^{\tau_-}} \\ \nonumber
& \approx |v'-v|^{-d+2+\gamma-\tau_-} \; |w|^{\tau_-}.
\end{align}
\end{itemize}
Hence, $B(r,\theta) \;  r^{-d+2}$ has the same behavior in both cases provided that
\begin{align}
-d+2+\gamma = \tau_+ + \tau_-.
\end{align}
Therefore, if $\tau_+ = -d+1+\nu$, which corresponds to the inverse power law \eqref{ipl}, 
then
\begin{align}\label{eq-two taus}
\tau_- = -d +2 +\gamma -(-d+1+\nu) = 1+\gamma+\nu,
\end{align}
which explains the choice in \eqref{eq-b} of the power for angles with negative cosine.

\bibliographystyle{abbrv}
\bibliography{Majabib}

\begin{thebibliography}{10}

\bibitem{al99}
R.~Alexandre.
\newblock Remarks on 3{D} {B}oltzmann linear equation without cutoff.
\newblock {\em Transport Theory Statist. Phys.}, 28(5):433--473, 1999.

\bibitem{aldeviwe00}
R.~Alexandre, L.~Desvillettes, C.~Villani, and B.~Wennberg.
\newblock Entropy dissipation and long-range interactions.
\newblock {\em Arch. Ration. Mech. Anal.}, 152(4):327--355, 2000.

\bibitem{amuxy11}
R.~Alexandre, Y.~Morimoto, S.~Ukai, C.-J. Xu, and T.~Yang.
\newblock The {B}oltzmann equation without angular cutoff in the whole space:
  {II}, {G}lobal existence for hard potential.
\newblock {\em Anal. Appl. (Singap.)}, 9(2):113--134, 2011.

\bibitem{amuxy12}
R.~Alexandre, Y.~Morimoto, S.~Ukai, C.-J. Xu, and T.~Yang.
\newblock Smoothing effect of weak solutions for the spatially homogeneous
  {B}oltzmann equation without angular cutoff.
\newblock {\em Kyoto J. Math.}, 52(3):433--463, 2012.

\bibitem{alvi02}
R.~Alexandre and C.~Villani.
\newblock On the {B}oltzmann equation for long-range interactions.
\newblock {\em Comm. Pure Appl. Math.}, 55(1):30--70, 2002.

\bibitem{alcagamo13}
R.~Alonso, J.~A. Ca{\~n}izo, I.~Gamba, and C.~Mouhot.
\newblock A new approach to the creation and propagation of exponential moments
  in the {B}oltzmann equation.
\newblock {\em Comm. Partial Differential Equations}, 38(1):155--169, 2013.

\bibitem{ar72I}
L.~Arkeryd.
\newblock On the {B}oltzmann equation. {I}. {E}xistence.
\newblock {\em Arch. Rational Mech. Anal.}, 45:1--16, 1972.

\bibitem{ar81}
L.~Arkeryd.
\newblock Intermolecular forces of infinite range and the {B}oltzmann equation.
\newblock {\em Arch. Rational Mech. Anal.}, 77(1):11--21, 1981.

\bibitem{ar83}
L.~Arkeryd.
\newblock {$L^{\infty }$} estimates for the space-homogeneous {B}oltzmann
  equation.
\newblock {\em J. Statist. Phys.}, 31(2):347--361, 1983.

\bibitem{bo84}
A.~V. Bobylev.
\newblock Exact solutions of the nonlinear {B}oltzmann equation and the theory
  of relaxation of a {M}axwell gas.
\newblock {\em Teoret. Mat. Fiz.}, 60(2):280--310, 1984.

\bibitem{bo97}
A.~V. Bobylev.
\newblock Moment inequalities for the {B}oltzmann equation and applications to
  spatially homogeneous problems.
\newblock {\em J. Statist. Phys.}, 88(5-6):1183--1214, 1997.

\bibitem{bogapa04}
A.~V. Bobylev, I.~M. Gamba, and V.~A. Panferov.
\newblock Moment inequalities and high-energy tails for {B}oltzmann equations
  with inelastic interactions.
\newblock {\em J. Statist. Phys.}, 116(5-6):1651--1682, 2004.

\bibitem{ca57}
T.~Carleman.
\newblock {\em Probl\`emes math\'ematiques dans la th\'eorie cin\'etique des
  gaz}.
\newblock Publ. Sci. Inst. Mittag-Leffler. 2. Almqvist \& Wiksells Boktryckeri
  Ab, Uppsala, 1957.

\bibitem{de93}
L.~Desvillettes.
\newblock Some applications of the method of moments for the homogeneous
  {B}oltzmann and {K}ac equations.
\newblock {\em Arch. Rational Mech. Anal.}, 123(4):387--404, 1993.

\bibitem{de95}
L.~Desvillettes.
\newblock About the regularizing properties of the non-cut-off {K}ac equation.
\newblock {\em Comm. Math. Phys.}, 168(2):417--440, 1995.

\bibitem{de97}
L.~Desvillettes.
\newblock Regularization properties of the {$2$}-dimensional non-radially
  symmetric non-cutoff spatially homogeneous {B}oltzmann equation for
  {M}axwellian molecules.
\newblock {\em Transport Theory Statist. Phys.}, 26(3):341--357, 1997.

\bibitem{dego00}
L.~Desvillettes and F.~Golse.
\newblock On a model {B}oltzmann equation without angular cutoff.
\newblock {\em Differential Integral Equations}, 13(4-6):567--594, 2000.

\bibitem{dewe04}
L.~Desvillettes and B.~Wennberg.
\newblock Smoothness of the solution of the spatially homogeneous {B}oltzmann
  equation without cutoff.
\newblock {\em Comm. Partial Differential Equations}, 29(1-2):133--155, 2004.

\bibitem{el83}
T.~Elmroth.
\newblock Global boundedness of moments of solutions of the {B}oltzmann
  equation for forces of infinite range.
\newblock {\em Arch. Rational Mech. Anal.}, 82(1):1--12, 1983.

\bibitem{emot53}
A.~Erd{\'e}lyi, W.~Magnus, F.~Oberhettinger, and F.~G. Tricomi.
\newblock {\em Higher transcendental functions. {V}ol. {III}}.
\newblock McGraw-Hill Book Company, Inc., New York-Toronto-London, 1955.
\newblock Based, in part, on notes left by Harry Bateman.

\bibitem{gapavi09}
I.~M. Gamba, V.~Panferov, and C.~Villani.
\newblock Upper {M}axwellian bounds for the spatially homogeneous {B}oltzmann
  equation.
\newblock {\em Arch. Ration. Mech. Anal.}, 194(1):253--282, 2009.

\bibitem{go97}
T.~Goudon.
\newblock On {B}oltzmann equations and {F}okker-{P}lanck asymptotics: influence
  of grazing collisions.
\newblock {\em J. Statist. Phys.}, 89(3-4):751--776, 1997.

\bibitem{gr63}
H.~Grad.
\newblock Asymptotic theory of the {B}oltzmann equation. {II}.
\newblock In {\em Rarefied {G}as {D}ynamics ({P}roc. 3rd {I}nternat. {S}ympos.,
  {P}alais de l'{UNESCO}, {P}aris, 1962), {V}ol. {I}}, pages 26--59. Academic
  Press, New York, 1963.

\bibitem{grst11}
P.~T. Gressman and R.~M. Strain.
\newblock Global classical solutions of the {B}oltzmann equation without
  angular cut-off.
\newblock {\em J. Amer. Math. Soc.}, 24(3):771--847, 2011.

\bibitem{li94}
P.-L. Lions.
\newblock On {B}oltzmann and {L}andau equations.
\newblock {\em Philos. Trans. Roy. Soc. London Ser. A}, 346(1679):191--204,
  1994.

\bibitem{lumo12}
X.~Lu and C.~Mouhot.
\newblock On measure solutions of the {B}oltzmann equation, part {I}: moment
  production and stability estimates.
\newblock {\em J. Differential Equations}, 252(4):3305--3363, 2012.

\bibitem{miwe99}
S.~Mischler and B.~Wennberg.
\newblock On the spatially homogeneous {B}oltzmann equation.
\newblock {\em Ann. Inst. H. Poincar\'e Anal. Non Lin\'eaire}, 16(4):467--501,
  1999.

\bibitem{mo06}
C.~Mouhot.
\newblock Rate of convergence to equilibrium for the spatially homogeneous
  {B}oltzmann equation with hard potentials.
\newblock {\em Comm. Math. Phys.}, 261(3):629--672, 2006.

\bibitem{si14}
L.~Silvestre.
\newblock A new regularization mechanism for the {B}oltzmann equation without
  cut-off.
\newblock {\em Comm. Math. Phys.}, 348(1):69--100, 2016.

\bibitem{algapata15}
M.~Taskovi\'{c}, R.~J. Alonso, I.~M. Gamba, and N.~Pavlovi\'{c}.
\newblock On {M}ittag-{L}effler moments for the {B}oltzmann equation for hard
  potentials without cutoff.
\newblock {\em arXiv:1512.06769}, 2015.

\bibitem{uk84}
S.~Ukai.
\newblock Local solutions in {G}evrey classes to the nonlinear {B}oltzmann
  equation without cutoff.
\newblock {\em Japan J. Appl. Math.}, 1(1):141--156, 1984.

\bibitem{vi98}
C.~Villani.
\newblock On a new class of weak solutions to the spatially homogeneous
  {B}oltzmann and {L}andau equations.
\newblock {\em Arch. Rational Mech. Anal.}, 143(3):273--307, 1998.

\bibitem{vi99}
C.~Villani.
\newblock Regularity estimates via the entropy dissipation for the spatially
  homogeneous {B}oltzmann equation without cut-off.
\newblock {\em Rev. Mat. Iberoamericana}, 15(2):335--352, 1999.

\bibitem{vi02}
C.~Villani.
\newblock A review of mathematical topics in collisional kinetic theory.
\newblock In {\em Handbook of mathematical fluid dynamics, {V}ol. {I}}, pages
  71--305. North-Holland, Amsterdam, 2002.

\bibitem{we94}
B.~Wennberg.
\newblock On moments and uniqueness for solutions to the space homogeneous
  {B}oltzmann equation.
\newblock {\em Transport Theory Statist. Phys.}, 23(4):533--539, 1994.

\bibitem{we94a}
B.~Wennberg.
\newblock Regularity in the {B}oltzmann equation and the {R}adon transform.
\newblock {\em Comm. Partial Differential Equations}, 19(11-12):2057--2074,
  1994.

\bibitem{we96}
B.~Wennberg.
\newblock The {P}ovzner inequality and moments in the {B}oltzmann equation.
\newblock In {\em Proceedings of the {VIII} {I}nternational {C}onference on
  {W}aves and {S}tability in {C}ontinuous {M}edia, {P}art {II} ({P}alermo,
  1995)}, number 45, part II, pages 673--681, 1996.

\bibitem{we97}
B.~Wennberg.
\newblock Entropy dissipation and moment production for the {B}oltzmann
  equation.
\newblock {\em J. Statist. Phys.}, 86(5-6):1053--1066, 1997.

\end{thebibliography}

\end{document}